\documentclass[10pt]{article}
\usepackage{geometry,amsmath,amssymb,amsthm,graphicx,epsfig}
\usepackage{subcaption,enumerate,dcolumn,latexsym,epstopdf,color}
\usepackage[graph]{xy}
\usepackage[colorlinks=true,linkcolor=blue,citecolor=blue]{hyperref}
\usepackage[usenames,dvipsnames,svgnames,table]{xcolor}
\usepackage{ulem}
\usepackage{hyperref}
\numberwithin{equation}{section}
\usepackage{multirow}
\usepackage{tabularx}

\usepackage{pdfsync,comment,a4wide}
\usepackage{framed}
\usepackage{enumitem}
\usepackage[titletoc,title]{appendix}
\newtheorem{theorem}{Theorem} 
\newtheorem{proposition}{Proposition} 
\newtheorem{lemma}{Lemma} 
\theoremstyle{definition}
\newtheorem{remark}{Remark} 
\newtheorem{definition}{Definition} 
\newtheorem{assumption}{Assumption}

\renewcommand{\em}[1]{\normalem \em{#1}}
\renewcommand{\emph}[1]{\normalem \emph{#1}}
\newcommand{\ie}{{\em i.e.}, }

\newcommand{\nn}{\mathbb{N}} 
\newcommand{\real}{\mathbb{R}} 
\newcommand{\norm}[1]{\left\Vert {#1} \right\Vert} 
\newcommand{\erl}{\left(-\infty , +\infty\right]} 

\DeclareMathOperator*{\argmin}{\arg\!\min}

\newcommand{\dist}{\mathrm{dist}} 
\newcommand{\prox}{\mathrm{prox}} 
\newcommand{\act}[1]{\left\langle {#1} \right\rangle} 
\newcommand{\seq}[2]{\left\{{#1}_{{#2}}\right\}_{{#2} \in \mathbb{N}}}
\newcommand{\Seq}[2]{\left\{{#1}^{{#2}}\right\}_{{#2} \in \mathbb{N}}}

\newcommand{\ba}{{\bf a}}
\newcommand{\bo}{{\bf 0}}

\newcommand{\bu}{{\bf u}}

\newcommand{\bx}{{\bf x}}
\newcommand{\by}{{\bf y}}
\newcommand{\bz}{{\bf z}}
\newcommand{\bb}{{\bf b}}

\newcommand{\bw}{{\bf w}}

\title{Convex Bi-Level Optimization Problems with Non-smooth Outer Objective Function}
\author{Roey Merchav\footnote{Faculty of Data and Decision Sciences, Technion--Israel Institute of Technology, Haifa 3200003, Israel. E-mail: merhav.roey@campus.technion.ac.il.} \and Shoham Sabach\footnote{Faculty of Data and Decision Sciences, Technion--Israel Institute of Technology, Haifa 3200003, Israel. E-mail: ssabach@technion.ac.il. This work was supported by the Israel Science Foundation (Grant 2480/21)}}
\date{\today}

\begin{document}
\maketitle

\begin{abstract}
	In this paper, we propose the Bi-Sub-Gradient (Bi-SG) method, which is a generalization of the classical sub-gradient method to the setting of convex bi-level optimization problems. This is a first-order method that is very easy to implement in the sense that it requires only a computation of the associated proximal mapping or a sub-gradient of the outer non-smooth objective function, in addition to a proximal gradient step on the inner optimization problem. We show, under very mild assumptions, that Bi-SG tackles bi-level optimization problems and achieves sub-linear rates both in terms of the inner and outer objective functions. Moreover, if the outer objective function is additionally strongly convex (still could be non-smooth), the outer rate can be improved to a linear rate. Last, we prove that the distance of the generated sequence to the set of optimal solutions of the bi-level problem converges to zero. 		
	
\end{abstract}

\section{Introduction} \label{Sec:Introduction}
	Bi-level optimization problems consist, in a hierarchical sense, of two optimization problems referred to as inner and outer problems. This class of problems is a very active area of research with practical applications in many different areas (for instance, in the field of machine learning). For a recent comprehensive review of bi-level optimization problems and applications, see \cite{DZ2020} and references therein. We also refer interested readers to the following recent papers \cite{AY2019,F2021,SNK2022} that discuss various interesting applications, which fit to the setting of this paper.
\medskip

	In this paper, we are interested in a specific class of what is often called simple bi-level optimization problems. More precisely, following the terminology of inner and outer problems, we are interested in the following bi-level optimization problems. The \emph{outer} level is given by the following convex constrained minimization problem
 	\begin{equation} \label{Prob:OP} \tag{OP}
		\min_{\bx \in X^{\ast}}  \omega\left(\bx\right),
 	\end{equation}
	where $X^{\ast}$ is the, assumed nonempty, set of minimizers of the \emph{inner} level problem. Here, in this paper, the inner problem is given by the classical convex composite model
	\begin{equation} \label{Prob:IP} \tag{IP}
		\min_{\bx \in \real^{n}} \left\{ \varphi\left(\bx\right) := f\left(\bx\right) + g\left(\bx\right) \right\},
	\end{equation}
	where $f$ is a continuously differentiable function and $g$ is an extended valued (possibly non-smooth) function. It should be noted that the outer objective function $\omega$ is only assumed to be convex (could be non-smooth).
\medskip

	In \cite{TA77-B}, Tikhonov proposed the first regularization method to deal with bi-level optimization problems, and it has evolved into the most popular approach to solve such problems. This approach does, however, have several limitations (see \cite{BS2014} for more details and relevant references). One of the major limitations is the determination of the ``right" regularization parameter. Accordingly, the performance of this technique is significantly impacted by this parameter. Therefore, in practice, the regularized problem is solved with multiple parameters in order to select one that works best. This implies an additional computational complexity that is increased with the number of parameters being tested.
\medskip

	This issue was first overcame by Solodov, who proposed a projected gradient method for the regularized problem. More precisely, in \cite{S07}, the following regularized problem (for some $\lambda > 0$),
	\begin{equation} \label{Prob:RP} \tag{P$_\lambda$}
		\min_{\bx \in \real^{n}}  \left\{ \varphi_{\lambda}\left(\bx\right) := \varphi\left(\bx\right) + \lambda\omega\left(\bx\right) \right\},
	\end{equation}
	was considered under the assumption that $\omega$ is smooth with a Lipschitz continuous gradient and $g$ (the non-smooth part of $\varphi$) is an indicator function of a certain closed and convex set. Starting with any $\bx^{0} \in \real^{n}$, by performing one step of projected gradient on (P$_{\lambda_{k}}$) for each $k \in \nn$, the distance of the generated sequence $\Seq{\bx}{k}$ to the set of optimal solutions of problem \eqref{Prob:OP}, was proved to converge to $0$. To this end, the sequence of regularization parameters $\seq{\lambda}{k}$ should satisfy that $\lambda_{k} \rightarrow 0$ as $k \rightarrow \infty$ and $\sum_{k = 1}^{\infty} \lambda_{k} = \infty$. Furthermore, subsequent studies \cite{S2007,HS2017} have delved into the realm of non-smooth bi-level optimization problems and established their convergence properties. For example, a bundle method \cite{S2007} and an $\varepsilon$-subgradient method \cite{HS2017} were analyzed in the setting of non-smooth bi-level problems. Apart from these three papers, the majority of research conducted on bi-level optimization problems has primarily focused on proving convergence results. For more papers on both smooth and non-smooth bi-level optimization problems, see \cite{X2011,YYY11,GRV2018,SVZ2021} and reference therein. However, there are only a few works (see more details below), which discuss the question of proving rate of convergence for algorithms tackling bi-level optimization problems.
\medskip

	This paper discusses first-order algorithms for solving bi-level optimization problems, which have a proven rate of convergence. Recently, some studies have proposed first-order algorithms for solving the problem \eqref{Prob:OP} with a proven rate of convergence result in terms of the inner problem \eqref{Prob:IP}. We will mention a few relevant works, whose first two motivated much of this research and are therefore described in detail below (see Section \ref{Sec:Bi-SG}). In \cite{BS2014}, which assumes that $g$ is an indicator function, the authors propose an algorithm with an inner rate of $O(1/\sqrt{k})$, where $k$ is the number of iterations. For an improved rate of $O(1/k)$, see \cite{SS2017}. It should be noted that both of these works are valid for the case that the outer objective function is smooth and strongly convex. More recently, in \cite{AY2019}, the authors deal with non-smooth outer objective functions, but the strong convexity is still needed. They propose an algorithm, which achieves an inner rate of $O(1/k^{0.5 - \varepsilon})$, where $0 < \varepsilon < 1/2$ is a parameter that is used in the algorithm.  Very recently, in \cite{KY2021}, the authors propose the first algorithm for tackling bi-level optimization problems\footnote{It should be noted that their setting of bi-level optimization problems is more rich since the inner level problem is given as a solution set of a variational inequality.} with non-smooth and convex (that is, not necessarily strongly convex) outer objective function, when the function $g$ of the inner problem is assumed to be an indicator function of a closed and convex set. The authors propose an algorithm that achieves the rates of convergence of $O(1/k^{1/2 - \alpha})$ and $O(1/k^{\alpha})$ in terms of the inner and outer objective functions, respectively, where $\alpha \in \left(0 , 1/2\right)$ is a hyper-parameter. It should be noted that in \cite{KY2021}, the authors also study a randomized version of the paper, due to the block setting of the inner problem. Even though we believe our results can be generalized to the block setting, we focus on the setting mentioned above to keep the paper concise and easy to follow. Two additional recent works that achieve rate of convergence results but in different settings are \cite{SNK2022} that study non-smooth bi-level optimization problems in the setting of online optimization and \cite{F2021} for distributed optimization in directed networks, see also references therein for more works on bi-level optimization problems in these domains.
\medskip

	The main goal of this paper, as described above, is to study bi-level optimization problems where the outer objective function $\omega(\cdot)$ is neither strongly convex nor smooth. Motivated by the recent work \cite{KY2021}, we tackle this challenge by developing a simple algorithm, called Bi-Sub-Gradient (Bi-SG), and proving that, if the outer objective function $\omega(\cdot)$ is convex, Bi-SG enjoys the rates of convergence of $O(1/k^{\alpha})$ and $O(1/k^{1-\alpha})$ in terms of the inner and outer objective functions, respectively, where $\alpha \in \left(1/2 , 1\right)$ is a hyper-parameter (when $\alpha = 1$ we obtain only an inner rate). Focusing on bi-level optimization problems where the inner problem is given as in \eqref{Prob:IP}, this work improves the rates obtained in \cite{KY2021} in two ways: (i) the rates in terms of the inner objective function are proven for the sequence itself and not the ergodic sequence and, (ii) the sub-linear rates are better. Moreover, our work deals with a more general class of inner optimization problem in comparison to \cite{KY2021}, which assumes that $g$ is an indicator function. In addition, we analyze the proposed Bi-SG algorithm in the special case where the outer objective function $\omega(\cdot)$ is strongly convex (but could be non-smooth). In this case, we prove that Bi-SG provides a linear rate of convergence in terms of the outer objective function $\omega(\cdot)$, while keeping the same inner rate of $O(1/k^{\alpha})$. We also prove that the distance of the generated sequence from the set of optimal solutions of the problem \eqref{Prob:OP} converges to $0$. 
\medskip

	The paper is organized in the following way. We begin with the next section that is devoted to the new notion of quasi Lipschitz mappings, which is a new property that is very helpful in our analysis and seems to have an importance by its own. In Section \ref{Sec:Bi-SG}, we discuss the optimization framework of the class of the studied bi-level problems, and describe in details the Bi-SG algorithm. In Section \ref{Sec:Theory}, we first give all of the notations and auxiliary results that are needed for the forthcoming results. We then prove rate of convergence results of Bi-SG in terms of both the inner and the outer objective functions. This section also includes our improved results in the case where the outer function is additionally strongly convex (still could be non-smooth). Section \ref{Sec:Numerical Experiments} contains numerical experiments comparing Bi-SG to existing methods and showing its computational superiority.
\medskip

	Throughout the paper we denote vectors by boldface letters. The notation $\act{\cdot , \cdot}$ is used to denote the inner product of two vectors and $\norm{\cdot}$ is the norm associated with this inner product, unless stated otherwise.

\section{The Quasi Lipschitz Property}
	In this section, we study the new property that we call quasi Lipschitz, which plays a key role in the main theoretical results to be developed below and could be of importance in other contexts of optimization.
	\begin{definition}[Quasi Lipschitz] \label{D:QuasiL}
 		Let $F : \real^{n} \rightarrow \real^{n}$ be a mapping. If there exist two non-negative constants $d_{1}, d_{2} \geq 0$ such that for any $\bx \in \real^{n}$, we have that $\norm{F\left(\bx\right)} \leq \max\left\{d_{1}, d_{2}\norm{\bx}\right\}$, then we say that $F$ is quasi Lipschitz with constants $\left(d_{1} , d_{2}\right)$.
 	\end{definition}
 	It should be noted that if a mapping is quasi Lipschitz with constants $\left(d_{1} , d_{2}\right)$, then it is also a quasi Lipschitz with constants $\left(d_{1}' , d_{2}'\right)$ for any constants $d_{1} \leq d_{1}'$ and $d_{2} \leq d_{2}'$.
\medskip
 
 	We begin with some simple properties of quasi Lipschitz mappings.
	\begin{proposition} \label{P:BasicQL}
		\begin{itemize}
			\item[$\rm{(i)}$] Let $\alpha \in \mathbb{R}$. If $F$ is quasi Lipschitz with constants $\left(d_{1} , d_{2}\right)$, then $\alpha F$ is quasi Lipschitz with constants $\left(|\alpha| d_{1} , |\alpha| d_{2}\right)$. 
    			\item[$\rm{(ii)}$] If $F_{1}$ and $F_{2}$ are quasi Lipschitz mappings with constants $\left(d_{1,1} , d_{2,1}\right)$ and $\left(d_{1,2} , d_{2,2}\right)$, then $F \equiv F_{1} + F_{2}$ is quasi Lipschitz with constants $\left(2\left(d_{1,1} + d_{1,2}\right) , 2\left(d_{2,1} + d_{2,2}\right)\right)$. In addition, the composite mapping $F \equiv F_{1} \circ F_{2}$ is quasi Lipschitz with constants $\left(2\left(d_{1 , 1} + d_{2 ,1}d_{1 , 2}\right) , 2d_{2 ,1}d_{2 , 2}\right)$\footnote{Both results remain valid for any finite number of mappings.}.
    			\item[$\rm{(iii)}$] Let $A : \real^{n} \rightarrow \real^{n}$ be a linear mapping. If $F$ is quasi Lipschitz with constants $\left(d_{1} , d_{2}\right)$, then $F \circ A$ is quasi Lipschitz with constants $\left(d_{1} , d_{2}\norm{A}\right)$.
		\end{itemize}
	\end{proposition} 
	\begin{proof}
    		\begin{itemize}
			\item[$\rm{(i)}$] For any $\bx \in \real^{n}$, from Definition \ref{D:QuasiL}, we have that
				\begin{equation*}    
					\norm{\alpha F\left(\bx\right)} = |\alpha|\norm{F\left(\bx\right)} \leq |\alpha| \max\{d_{1} , d_{2}\norm{\bx}\} = \max\{|\alpha| d_{1} , |\alpha| d_{2}\norm{\bx}\},
				\end{equation*}    
    				which proves the desired result.    
    			\item[$\rm{(ii)}$] For any $\bx \in \real^{n}$, we have for $i = 1 , 2$ that
    				\begin{equation*}    
        				\norm{F_{i}\left(\bx\right)} \leq \max \left\{d_{1 , i} , d_{2 , i}\norm{\bx}\right\} \leq d_{1 , i} + d_{2 , i}\norm{\bx},
				\end{equation*}    
    				where the first inequality follows from Definition \ref{D:QuasiL} and the last inequality follows from the fact that $\max\{a , b\} \leq a + b$ for any $a , b \geq 0$. Therefore
    				\begin{equation*}    
        				\norm{F\left(\bx\right)} \leq \norm{F_{1}\left(\bx\right)} + \norm{F_{2}\left(\bx\right)} \leq \left(d_{1 , 1} + d_{1 , 2}\right) + \left(d_{2 , 1} + d_{2 , 2}\right)\norm{\bx},
				\end{equation*}    
    				and the first desired result now follows from the fact that $a + b \leq \max\{2a , 2b\}$ for any $a , b \geq 0$. For any $\bx \in \real^{n}$, from Definition \ref{D:QuasiL}, we have that
				\begin{align*}    
        				\norm{F\left(\bx\right)} & \leq \max\left\{ d_{1 , 1} , d_{2 , 1}\norm{F_{2}\left(\bx\right)}\right\} \\
     			   	& \leq \max\left\{ d_{1 , 1} , d_{2 , 1}\max\left\{ d_{1 , 2} , d_{2 , 2}\norm{\bx} \right\} \right\} \\
        				& = \max\left\{ d_{1 , 1} , d_{2 ,1}d_{1 , 2} , d_{2 ,1}d_{2 , 2}\norm{\bx} \right\}.
				\end{align*}    
    				Following the same arguments as above we get the second desired result.
    			\item[$\rm{(iii)}$] For any $\bx \in \real^{n}$, from Definition \ref{D:QuasiL}, we have that
				\begin{equation*}    
    					\norm{F\left(A\bx\right)} \leq \max\{d_{1} , d_{2}\norm{A\bx}\} \leq \max\{d_{1} , d_{2}\norm{A}\norm{\bx}\},
				\end{equation*}    
    				which proves the desired result.    
		\end{itemize}\vspace{-0.2in}
	\end{proof}
	Now, we provide two important examples of quasi Lipschitz mappings that are very relevant from the optimization point of view.
 	\begin{proposition} \label{P:ExQuasiL}
		The following statements hold:
		\begin{itemize}
    			\item[$\rm{(i)}$] If $h : \real^{n} \rightarrow \real$ is convex and globally $L$-Lipschitz continuous, then any sub-gradient of $h$ is quasi Lipschitz with constants $\left(L , 0\right)$.
    			\item[$\rm{(ii)}$] If $T : \real^{n} \rightarrow \real^{n}$ is a globally $L$-Lipschitz mapping, that is, $\norm{T\left(\bx\right) - T\left(\by\right)} \leq L\norm{\bx - \by}$ for all $\bx , \by \in \real^{n}$, then it is quasi Lipschitz with constants $\left(2\norm{T\left(\bo\right)} , 2L\right)$. In particular, any nonexpansive mapping is quasi Lipschitz with constants $\left(2\norm{T\left(\bo\right)} , 2\right)$.					
    			\item[$\rm{(iii)}$] If $h : \real^{n} \rightarrow \real$ is continuously differentiable with a Lipschitz continuous gradient with constant $L_{h}$, then $\nabla h$ is quasi Lipschitz with constants $\left(2\norm{\nabla h\left(\bo\right)} , 2L_{h}\right)$.
    		\end{itemize}
    \end{proposition}
    \begin{proof}
    		\begin{itemize}
    			\item[$\rm{(i)}$] Let $\bx \in \real^{n}$. First, pick some $\xi\left(\bx\right) \in \partial h\left(\bx\right)$. From the sub-gradient inequality and the definition of $L$, we get that
				\begin{equation*}
    					\norm{\xi\left(\bx\right)}^{2} = \xi\left(\bx\right)^{T}\left(\bx + \xi\left(\bx\right) - \bx\right) \leq h\left(\bx + \xi\left(\bx\right)\right) - h\left(\bx\right) \leq L\norm{\xi\left(\bx\right)},
				\end{equation*}    
    				which directly implies that $\xi$ is quasi Lipschitz with $d_{1} = L$ and $d_{2} = 0$, as desired.
     		\item[$\rm{(ii)}$] Let $\bx \in \real^{n}$. Since $T$ is globally $L$-Lipschitz, we get that
    				\begin{equation*}
        				\norm{T\left(\bx\right)} \leq \norm{T\left(\bx\right) - T\left(\bo\right)} + \norm{T\left(\bo\right)} \leq L\norm{\bx} + \norm{T\left(\bo\right)}.
				\end{equation*}   
    				Hence, the desired result now follows from the fact that $a + b \leq \max\{2a , 2b\}$ for any $a , b \geq 0$.
    			\item[$\rm{(iii)}$] The result follows immediately from item (ii) with $T \equiv \nabla h$ and $L = L_{h}$.
    		\end{itemize}\vspace{-0.2in}
    \end{proof}
	Now, in order to demonstrate how useful the quasi Lipschitz property is, we would like to mention that there are some important classes of functions that have quasi Lipschitz sub-gradients but are neither smooth nor globally Lipschitz continuous. To this end, in addition to the basic properties of quasi Lipschitz mappings that we have proven above (see Proposition \ref{P:BasicQL}), the following chain rule property will be useful.
	\begin{lemma}[Composition rule for quasi Lipschitz mappings]
		Let $D \subset \real$ and $\psi : \real^{n} \rightarrow D$ be a globally $L$-Lipschitz continuous and convex function. In addition, let $\phi : \real \rightarrow \real$ be a non-decreasing and differentiable function on $D$, such that $\phi$ has a quasi Lipschitz derivative with constants $\left(d_{1} , d_{2}\right)$. Then, any sub-gradient of $h := \phi \circ \psi$ is quasi Lipschitz with constants $\left(2\left(Ld_{1} + Ld_{2}|\psi\left(\bo\right)|\right) , 2L^{2}d_{2}\right)$.
	\end{lemma}
	\begin{proof}
    		 Let $\bx \in \real^{n}$. From the chain rule for sub-differentials (see \cite[page 287]{B2017-B}), we have that $\partial h\left(\bx\right) = \phi'\left(\psi\left(\bx\right)\right)\partial \psi\left(\bx\right)$. Thus, by denoting $\xi\left(\bx\right) := \phi'\left(\psi\left(\bx\right)\right)\eta\left(\bx\right)$ for some $\eta\left(\bx\right) \in \partial \psi\left(\bx\right)$, we obtain that $\xi\left(\bx\right) \in \partial h\left(\bx\right)$. Now, using the facts that $\phi'$ is quasi Lipschitz and $\psi$ is globally $L$-Lipschitz continuous, we obtain
    		\begin{align*}
        		\norm{\xi\left(\bx\right)} & = \norm{\phi'\left(\psi\left(\bx\right)\right)\eta\left(\bx\right)} \\
        		& = \left|\phi'\left(\psi\left(\bx\right)\right)\right| \cdot \norm{\eta\left(\bx\right)} \\
        		& \leq \max \left\{d_{1} , d_{2}|\psi\left(\bx\right)|\right\} \cdot L \\
        		& \leq \max\left\{Ld_{1} , Ld_{2}\left(|\psi\left(\bx\right) - \psi\left(\bo\right)| + |\psi\left(\bo\right)|\right)\right\} \\
        		& \leq \max\left\{Ld_{1}, L^{2}d_{2}\norm{\bx} + Ld_{2}|\psi\left(\bo\right)|\right\} \\
        		& \leq Ld_{1} + Ld_{2}|\psi\left(\bo\right)| + L^{2}d_{2}\norm{\bx} \\
        		& \leq \max\left\{2\left(Ld_{1} + Ld_{2}|\psi\left(\bo\right)|\right), 2L^{2}d_{2}\norm{\bx}\right\},
    		\end{align*}
    		where the first inequality follows from the fact that $\norm{\eta\left(\bx\right)} \leq L$ due to the fact that $\psi$ is globally $L$-Lipschitz continuous (see \cite[Theorem 3.61, p. 71]{B2017-B}). This implies the desired result. 
	\end{proof}
	For example, the squared $\ell_{1}$-norm, \ie $h\left(\bx\right) = \norm{\bx}_{1}^{2}$, which is not differentiable and not globally Lipschitz continuous, is quasi Lipschitz with constants $\left(0 , 4n\right)$ due to the lemma above, since $\psi \equiv \norm{\cdot}_{1}$ is globally Lipschitz continuous with a constant $L = \sqrt{n}$ and $\phi\left(t\right) = t^{2}$ is quasi Lipschitz with constants $\left(0 , 2\right)$. Many more examples can be generated using the properties we have proven in the results above.
\medskip

	A final obvious remark is about a boundedness property of quasi Lipschitz mappings.
	\begin{remark} \label{R:BounQL}
		Let $F : \real^{n} \rightarrow \real^{n}$ be a quasi Lipschitz mapping. Then, $F$ is bounded on any bounded subset of $\real^{n}$.	
	\end{remark}
	
\section{The Bi-Sub-Gradient Algorithm} \label{Sec:Bi-SG}
	In this paper, we are focusing on bi-level optimization problems which are formulated above using the problems couple \eqref{Prob:IP} and \eqref{Prob:OP}. We first discuss the inner level problem which is given by
	\begin{equation} \tag{IP}
		\min_{\bx \in \real^{n}} \left\{ \varphi\left(\bx\right) := f\left(\bx\right) + g\left(\bx\right) \right\}.
	\end{equation}
	Problem \eqref{Prob:IP}, which consists of minimizing the sum of a smooth function $f$ and a possibly nonsmooth function $g$, is one of the most studied models in modern optimization with a huge body of literature (see, for instance, \cite{BT10} and the references therein). We have adopted this structure of the inner problem since it is a very flexible optimization model that was also studied in the context of bi-level optimization (see, for instance, \cite{BS2014,SS2017} and references therein). It should be noted that, even though the results in \cite{BS2014} could be generalized to this model, the authors dealt in their paper with the particular case where $g$ is an indicator function of a certain closed and convex set. Following these previous works, we make the following standing assumption on the functions of the inner level problem \eqref{Prob:IP} (see the recent book \cite{B2017-B} for more details and information on this model as a stand alone optimization problem).
	\begin{assumption} \label{A:Composite}
		\begin{itemize}
			\item[$\rm{(i)}$] $f : \real^{n} \rightarrow \real$ is convex and continuously differentiable such that its gradient is Lipschitz with constant $L_{f}$, that is,
				\begin{equation*}
					\norm{\nabla f\left(\bx\right) - \nabla f\left(\by\right)} \leq L_{f}\norm{\bx - \by}, \quad \forall \,\, \bx , \by \in \real^{n}.
				\end{equation*}
			\item[$\rm{(ii)}$] $g : \real^{n} \rightarrow \left(-\infty , \infty\right]$ is proper, lower semicontinuous and convex.
			\item[$\rm{(iii)}$] The set $X^{\ast}$ of all optimal solutions of problem \eqref{Prob:IP} is nonempty, that is, $X^{\ast} \neq \emptyset$.
		\end{itemize}
	\end{assumption}
	The basic algorithm for solving problem \eqref{Prob:IP}, as a stand alone optimization problem, is the so called \emph{Proximal Gradient} (PG), which at a given point $\bx \in \real^{n}$ generates the next point via the following rule
	\begin{equation} \label{A:PG-Step}
		\bx^{+} = \prox_{tg}\left(\bx - t\nabla f\left(\bx\right)\right),
	\end{equation}
	for some step-size $t > 0$. The main operation of this algorithm is the computation of the \emph{Moreau proximal mapping} of a proper, lower semicontinuous and convex function $h : \real^{n} \rightarrow \left(-\infty , \infty\right]$, which is denoted and defined by
	\begin{equation}
		\prox_{h}\left(\bx\right) := \argmin_{\bu \in \real^{n}} \left\{ h\left(\bu\right) + \frac{1}{2}\norm{\bu - \bx}^{2} \right\}.
	\end{equation}
	It will be helpful to recall the \emph{prox-grad mapping} that is defined by
	\begin{equation} \label{ProxGradMap}
		T_{t}^{f , g}\left(\bx\right) := \prox_{tg}\left(\bx - t\nabla f\left(\bx\right)\right).
	\end{equation}
	Another mapping that is very connected to the prox-grad mapping and will be essential in our paper is the gradient mapping that is defined as follows
	\begin{equation} \label{GradMap}
     	G_{t}^{f , g}\left(\bx\right) := \frac{1}{t}\left(\bx - T_{t}^{f , g}\left(\bx\right)\right).
	\end{equation}
	For simplicity, in both mappings, we will omit the superscripts $f , g$ and simply write $T_{t}$ and $G_{t}$ instead of $T_{t}^{f , g}$ and $G_{t}^{f , g}$, respectively, whenever no confusion arises.
\medskip

	As mentioned above, the papers \cite{BS2014,SS2017} motivated the development of the Bi-Sub-Gradient (Bi-SG) algorithm, which as the two previous algorithms tackle the inner problem using a proximal gradient step (see \eqref{A:PG-Step}). However, the main difference between the algorithms of \cite{BS2014,SS2017} to the Bi-SG that we propose, is in the way that the outer optimization is handled. Therefore, we would like first to describe how \cite{BS2014,SS2017} proposed to tackle the outer problem. To this end, we recall the outer optimization problem
	\begin{equation} \tag{OP}
		\min_{\bx \in X^{\ast}} \omega\left(\bx\right).
	\end{equation}	
	In \cite{BS2014}, the authors proposed the Minimal Norm Gradient (MNG) algorithm, which is based on cutting planes idea, and designed to tackle bi-level optimization problems with smooth and strongly convex outer objective functions $\omega\left(\cdot\right)$. At each iteration, two half-spaces are constructed, and the outer objective function $\omega\left(\cdot\right)$ is minimized over their intersection. The authors propose an algorithm, which at a given point $\bx \in \real^{n}$, uses the prox-grad output vector $T_{t}\left(\bx\right)$ (with a certain step-size) to build a half-space $W_{\bx}$. The second half-space $Q_{\bx}$ is built using the vector $\nabla \omega\left(\bx\right)$. Therefore, each iteration of the MNG algorithm consists of three computational tasks: (i) computing a prox-grad step of the inner problem, (ii) computing the gradient of the outer objective function, and (iii) solving
	\begin{equation*}
		\min \left\{ \omega\left(\by\right) : \, \by \in W_{\bx} \cap Q_{\bx} \right\}.
	\end{equation*}
	In \cite{SS2017}, the authors proposed the BiG-SAM algorithm, which is designed to tackle bi-level optimization problems for which the outer objective functions $\omega\left(\cdot\right)$ is smooth with a Lipschitz continuous gradient and strongly convex. The algorithm computes the next iteration by taking the convex combination of a proximal gradient step on the inner objective function and a gradient step on the outer objective function. Therefore, the BiG-SAM is cheaper to implement than the MNG, due to the third computational task of MNG (the first two tasks are equivalent) that might requires an additional nested algorithm in tackling it.
\medskip
	
	In this paper, we completely depart from the setting of smooth and strongly convex outer objective functions and propose an algorithm for bi-level optimization problems, which enjoys from theoretical guarantees (as we will prove later) in this setting. Before discussing our algorithm we will record the following mild assumption on the outer objective function that is required in our study. To this end, we recall that $h : \real^{n} \rightarrow \left(-\infty , \infty\right]$ is called a coercive function if $h\left(\bx\right) \rightarrow \infty$ as $\norm{\bx} \rightarrow \infty$.
	\begin{assumption} \label{A:Outer}
		$\omega : \real^{n} \rightarrow \real$ is convex and coercive. 
	\end{assumption}	
	As mentioned above, this study was motivated by \cite{BS2014,SS2017} and similarly to the MNG and BiG-SAM algorithms, we also compute a proximal gradient step on the inner optimization problem. However, regarding the outer problem, since we are dealing with non-smooth functions, we propose two different algorithms according to the structure of the outer objective function $\omega\left(\cdot\right)$. 
\medskip

	{\bf The outer objective function $\omega$ has no special structure.} In this case, we propose to use a sub-gradient step on the outer optimization problem. Moreover, all we need in this case is that at any point there exists a sub-gradient of $\omega\left(\cdot\right)$ that is quasi Lipschitz (as defined in Definition \ref{D:QuasiL}). It should be noted that this choice is relevant, for example, when the function $\omega\left(\cdot\right)$ is globally Lipschitz continuous (due to Proposition \ref{P:ExQuasiL}(i)). Therefore, in this case, the Bi-Sub-Gradient (Bi-SG) algorithm is given by
\vspace{0.2in}

	\hspace{0.1in}\fbox{\parbox{14.5cm}{{\bf Bi-Sub-Gradient - Version I}
		\begin{itemize}[label={}]
       		\item {\bf Input:} $\alpha \in \left(1/2 , 1\right]$, $c \in \left(0 , 1\right]$, and $\eta_{k} = c\left(k + 1\right)^{-\alpha}$  for all $k \in \nn$.
            \item {\bf Initialization:} $\bx^{0} \in \real^{n}$.
            \item {\bf General Step} ($k = 1 , 2 , \ldots$):
				\begin{itemize}
					\item[1.] Pick a step-size $t_{k} = 1/L_{k} > 0$ and compute
						\begin{equation} \label{V1:Inner}
							\by^{k} = T_{t_{k}}^{f , g}\left(\bx^{k}\right) \equiv \prox_{t_{k}g}\left(\bx^{k} - t_{k}\nabla f\left(\bx^{k}\right)\right).
						\end{equation}
					\item[2.] Pick a sub-gradient $\bz^{k} \in \partial \omega\left(\by^{k}\right)$ and compute 
						\begin{equation} \label{V1:Outer}
							\bx^{k + 1} = \by^{k} - \eta_{k}\bz^{k}.
						\end{equation}
        			\end{itemize}
        	\end{itemize}\vspace{-0.15in}}}
\vspace{0.2in}

	{\bf The outer objective function $\omega$ has a composite structure.} In this case, we assume that $\omega \equiv \sigma + \psi$, where $\sigma : \real^{n} \rightarrow \real$ is convex and continuously differentiable such that its gradient is Lipschitz with constant $L_{\sigma}$ and $\psi : \real^{n} \rightarrow \real$ is proper, lower semicontinuous and convex. Hence, in order to exploit this additional structure of $\omega\left(\cdot\right)$, we propose to apply a proximal gradient step on the outer objective function and thus the Bi-Sub-Gradient (Bi-SG) algorithm is given by
\vspace{0.2in}

	\hspace{0.1in}\fbox{\parbox{14.5cm}{{\bf Bi-Sub-Gradient - Version II}
		\begin{itemize}[label={}]
       		\item {\bf Input:} $\alpha \in \left(1/2 , 1\right]$, $c \leq \min\left\{ 1/ L_{\sigma} , 1\right\}$, and $\eta_{k} = c\left(k + 1\right)^{-\alpha}$  for all $k \in \nn$.
            \item {\bf Initialization:} $\bx^{0} \in \real^{n}$.
            \item {\bf General Step} ($k = 1 , 2 , \ldots$):
				\begin{itemize}
					\item[1.] Pick a step-size $t_{k} = 1/L_{k} > 0$ and compute
						\begin{equation} \label{V2:Inner}
							\by^{k} = T_{t_{k}}^{f , g}\left(\bx^{k}\right) \equiv \prox_{t_{k}g}\left(\bx^{k} - t_{k}\nabla f\left(\bx^{k}\right)\right).
						\end{equation}
					\item[2.] Compute 
						\begin{equation} \label{V2:Outer}
							\bx^{k + 1} = T_{\eta_{k}}^{\sigma , \psi}\left(\by^{k}\right) \equiv \prox_{\eta_{k}\psi}\left(\by^{k} - \eta_{k}\nabla \sigma\left(\by^{k}\right)\right).
						\end{equation}
        			\end{itemize}
        	\end{itemize}\vspace{-0.1in}}}
\vspace{0.2in}

	Before presenting the convergence results, we will conclude this part with a discussion on the classical ways to choose the step-size $L_{k}$ in step 1 of the Bi-SG algorithm (relevant for both versions).
	\begin{itemize}
    	\item[$\rm{(i)}$] \textbf{A constant step-size.} $L_{k} = L_{f}$ (see Assumption \ref{A:Composite}(i)).
    	\item[$\rm{(ii)}$] \textbf{A backtracking procedure.} The classical procedure (see, for instance, \cite{B2017-B}) requires two parameters $\gamma > 0$ and $\eta > 1$. Define $L_{-1} = \gamma$. At iteration $k \geq 0$, the choice of $L_{k}$ is done as follows. First, $L_{k}$ is set to be equal to $L_{k - 1}$. Then, while (recall that $t_{k} = 1 / L_{k}$)
    		\begin{equation} \label{BT:1}
    			f\left(T_{t_{k}}^{f , g}\left(\bx_{k}\right)\right) > f\left(\bx_{k}\right) + \act{\nabla f\left(\bx_{k}\right) , T_{t_{k}}^{f , g}\left(\bx_{k}\right) - \bx_{k}} + \frac{t_{k}}{2}\norm{T_{t_{k}}^{f , g}\left(\bx_{k}\right) - \bx_{k}}^{2},
    		\end{equation}
    		we set $L_{k} := \eta L_{k}$. In other words, $L_{k}$ is chosen as $L_{k} = L_{k - 1}\eta^{i_{k}}$, where $i_{k}$ is the smallest non-negative integer for which the condition \eqref{BT:1} is satisfied.
	\end{itemize}
	Now, we recall the following result (see \cite[Remark 10.19, p. 283]{B2017-B} that we will frequently use in the analysis below.
	\begin{proposition} \label{P:BackT}
		For  any $k \in \mathbb{N}$, we have that $\underline{L} \leq L_{k} \leq {\bar L}$, where $\underline{L} = {\bar L} = L_{f}$ for the constant step-size choice, while for the backtracking choice $\underline{L} = \gamma$ and ${\bar L} = \max \left\{ L_{f}\eta , \gamma \right\}$.
	\end{proposition}

\section{Convergence Analysis of Bi-SG} \label{Sec:Theory}
	In order to make the analysis clearer and simpler to read, we will first recall some classical results on proximal and prox-grad mappings. The rest of this section will be split into three parts. We begin with some basic properties of the Bi-SG algorithm and a boundedness result of the generated sequence. Then, we will provide a rate of convergence analysis of the Bi-SG algorithm, which will include results on both the inner and the outer objective functions. The section will be concluded with another rate of convergence result for the case where the outer objective function $\omega\left(\cdot\right)$ is additionally strongly convex (but still could be non-smooth).
\medskip

	In order to unify the convergence analysis of both versions, we will use the following unified version of the Bi-SG Algorithm that captures both cases. To this end, we will define a new mapping $\Omega_{k} : \real^{n} \rightarrow \real^{n}$, $k \in \nn$, as follows
	\begin{itemize}
    		\item[$\rm{(V1)}$] {\bf Version I}. $\Omega_{k} \equiv \Omega$, for any $k \in \nn$, where $\Omega: \real^{n} \rightarrow \real^{n}$ is a sub-gradient of $\omega$.
    		\item[$\rm{(V2)}$] {\bf Version II}. $\Omega_{k} \equiv G_{\eta_{k}}^{\sigma , \psi}$, \ie the gradient mapping corresponds to the composite function $\omega \equiv \sigma + \psi$.
	\end{itemize}
\vspace{0.2in}

	\hspace{0.1in}\fbox{\parbox{14.5cm}{{\bf Unified Bi-SG}
		\begin{itemize}[label={}]
       		\item {\bf Input:} $\alpha \in \left(1/2 , 1\right]$, $c \in \left(0 , 1\right]$, and $\eta_{k} = c\left(k + 1\right)^{-\alpha}$  for all $k \in \nn$.
            \item {\bf Initialization:} $\bx^{0} \in \real^{n}$.
            \item {\bf General Step} ($k = 1 , 2 , \ldots$):
				\begin{itemize}
					\item[1.] Pick a step-size $t_{k} = 1/L_{k} > 0$ and compute
						\begin{equation} \label{U:Inner}
							\by^{k} = T_{t_{k}}^{f , g}\left(\bx^{k}\right) \equiv \prox_{t_{k}g}\left(\bx^{k} - t_{k}\nabla f\left(\bx^{k}\right)\right).
						\end{equation}
					\item[2.] Pick a mapping $\Omega_{k}$ and compute 
						\begin{equation} \label{U:Outer} 
							\bx_{k + 1} = \by_{k} - \eta_{k}\Omega_{k}\left(\by_{k}\right).
						\end{equation}
        			\end{itemize}
        	\end{itemize}\vspace{-0.1in}}}
\vspace{0.2in}
	
	It should be noted that according to the version of the chosen Bi-SG Algorithm, the following assumption on the outer objective function $\omega\left(\cdot\right)$ should be made (in addition to Assumption \ref{A:Outer}).
	\begin{assumption} \label{A:Outer2}
		In the case of Version I only (C1) is needed and for Version II only (C2):
		\begin{itemize}
			\item[$\rm{(C1)}$] At each point there exists a subgradient of $\omega$ that is quasi Lipschitz. 
			\item[$\rm{(C2)}$] $\omega \equiv \sigma + \psi$, where $\sigma$ is convex and continuously differentiable such that its gradient is Lipschitz continuous with a constant $L_{\sigma}$, and $\psi : \real^{n} \rightarrow \real$ is proper, lower semicontinuous and convex.
		\end{itemize}
	\end{assumption}	
	It should be noted that Assumption C1, that is relevant only for Version I of the algorithm, is needed only for the sake of proving that the algorithm generates a bounded sequence (see Section \ref{SSec:Bound} below). Therefore, if the boundedness can be guaranteed based on other properties of the outer objective function $\omega$, the quasi Lipschitz assumption can be removed. Moreover, this assumption is not needed to derive the rate of convergence results as can be seen in Section \ref{SSec:Rate} below. 
	
\subsection{Mathematical Toolbox}
	Now we recall some fundamental properties on proximal and prox-grad mappings that will be used later, which are classical and can be found, for example, in \cite[Theorem 10.16, p. 281]{B2017-B}.
	\begin{proposition} \label{P:ProximalInequality}
		Let $s : \real^{n} \rightarrow \real$ be a continuously differentiable and convex function such that its gradient is Lipschitz continuous with a constant $L_{s} > 0$, and let $h : \real^{n} \rightarrow \erl$ be a proper, lower-semicontinuous and convex function. Then, for any $\bx , \by \in \real^{n}$ and $t \in \left(0, 1/L_{s}\right]$, it holds that
		\begin{align}
   				\phi\left(\bx - tG_{t}\left(\bx\right)\right) - \phi\left(y\right) & \leq \frac{1}{2t}\left(\norm{\bx - \by}^{2} - \norm{\left(\bx - t G_{t}\left(\bx\right)\right) - \by}^{2} \right) - \ell_{s}\left(\by , \bx\right) \label{P:ProximalInequality:1}\\
  				& \leq \frac{1}{2t}\left(\norm{\bx - \by}^{2} - \norm{\left(\bx - tG_{t}\left(\bx\right) \right) - \by}^{2}\right) \label{P:ProximalInequality:2},
		\end{align}
		where $\phi = s + h$, $G_{t}$ is the corresponding gradient mapping (see \ref{GradMap}), and $\ell_{s}\left(\by , \bx\right) = s\left(\by\right) - s\left(x\right) - \nabla s\left(x\right)^{T}\left(\by - \bx\right)$. In particular, we have that
		\begin{equation} \label{P:ProximalInequality:3}
			\phi\left(\bx - tG_{t}\left(\bx\right)\right) - \phi\left(\by\right) \leq G_{t}\left(\bx\right)^{T}\left(\bx - \by\right) - \frac{t}{2}\norm{G_{t}\left(\bx\right)}^{2}.
		\end{equation}
	\end{proposition}
	The following technical result will be essential below in deriving the aforementioned results.
	\begin{proposition} \label{P:SubGradProperty}
		Let $h : \real^{n} \rightarrow \erl$ be a proper, lower-semicontinuous and convex function. Then, for any $\bx \in \real^{n}$ and $t > 0$, we have that $\bx - \prox_{th}\left(\bx\right) \in t\partial h \left(\prox_{th}\left(\bx\right)\right)$.
	\end{proposition}
	We conclude this part with the following technical result that will be helpful in obtaining the promised rate of convergence results.
	\begin{lemma} \label{L:Sum}
		Let $n_{1} , n_{2} \in \nn$ satisfy that $n_{1} \leq n_{2}$. Then, for any $0 < r < 1$ we have
		\begin{equation*}
    			\sum_{n = n_{1}}^{n_{2}} n^{-r} \leq \frac{n_{2}^{1 - r} - \left(n_{1} - 1\right)^{1 - r}}{1 - r}.
		\end{equation*}
 		The result remains true for $r > 1$ when $n_{1} \geq 2$. 
 	\end{lemma}
	\begin{proof} 
		Since $t \rightarrow t^{-r}$ is a monotonically non-increasing function on $\real_{++}$, it follows for any $n \geq 1$ that $n^{-r} \leq \int_{t = n - 1}^{n} t^{-r}dt$. Thus
		\begin{equation*}
        		\sum_{n = n_{1}}^{n_{2}} n^{-r} \leq \int_{t = n_{1} - 1}^{n_{2}} t^{-r}dt = \frac{1}{1 - r}\left[t^{1 - r} \right]^{n_{2} }_{n_{1} - 1} = \frac{n^{1 - r}_{2} - (n_{1} - 1)^{1 - r}}{1 - r},
    		\end{equation*}
		which proves the desired result.
	\end{proof}

\subsection{Boundedness of Bi-SG} \label{SSec:Bound}
	The main goal in this section is to show that both versions of Bi-SG generate a bounded sequence. To this end and in order to simplify the analysis to come, we summarize first some technical basic properties of Bi-SG. Throughout the paper, we denote by $X'$ the set of all optimal solutions of the outer optimization problem \eqref{Prob:OP}.
	\begin{lemma}[Basic properties of Bi-SG] \label{L:Basic}
		Let $\Seq{\bx}{k}$ and $\Seq{\by}{k}$ be sequences generated by the Unified Bi-SG Algorithm. Then,
		\begin{itemize}
    			\item[$\rm{(i)}$] $\norm{\by^{k} - \bx^{\ast}} = \norm{T_{t_{k}}^{f , g}\left(\bx^{k}\right) - \bx^{\ast}} \leq \norm{\bx^{k} - \bx^{\ast}}$ for any $\bx^{\ast} \in X^{\ast}$ and $k \in \nn$.
    			\item[$\rm{(ii)}$] $\norm{\by^{k} - \bx'} \leq \norm{\bx^{k} - \bx'}$ for any $\bx' \in X'$ and $k \in \nn$.
    			\item[$\rm{(iii)}$] If $\sup \left\{ \norm{\bx^{k} - \bx'} : \, k \in \nn \right\} < \infty,$ then $\sup \left\{ \norm{\Omega_{k}\left(\by^{k}\right)} : \, k \in \nn \right\} < \infty$. 
		\end{itemize}
	\end{lemma}
	\begin{proof}
		\begin{itemize}
    			\item[$\rm{(i)}$] Let $\bx \in \real^{n}$ and $\bx^{\ast} \in X^{\ast}$. Applying Proposition \ref{P:ProximalInequality} with $s = f$, $h = g$, $t = t_{k}$, $\bx = \bx^{k}$ and $\by = \bx^{\ast}$, it follows from \eqref{P:ProximalInequality:2} that
				\begin{equation*}
    					\norm{T_{t_{k}}^{f , g}\left(\bx^{k}\right) - \bx^{\ast}}^{2} \leq \norm{\bx^{k} - \bx^{\ast}}^{2} + 2t_{k}\left(\varphi\left(\bx^{\ast}\right) - \varphi\left(T_{t_{k}}^{f , g}\left(\bx^{k}\right)\right)\right) \leq \norm{\bx^{k} - \bx^{\ast}}^{2},
				\end{equation*}
				where the last inequality follows from the fact that $\bx^{\ast}$ is a minimizer of $\varphi$. Thus, from step 1 of Bi-SG (see \eqref{U:Inner}), the desired result follows.
    			\item[$\rm{(ii)}$] For any $k \in \nn$, we apply item (i) with $\bx^{\ast} = \bx' \in X' \subseteq X^{\ast}$ to get that
				\begin{equation*}    
					\norm{\by^{k} - \bx'} = \norm{T_{t_{k}}^{f , g}\left(\bx^{k}\right) - \bx'} \leq \norm{\bx^{k} - \bx'}.
				\end{equation*}    			
			\item[$\rm{(iii)}$] Assume that $\sup \left\{ \norm{\bx_{k} - \bx'} : \, k \in \nn \right\} < \infty$. Therefore, from item (ii) we get that $\Seq{\by}{k}$ is bounded. Now, we split the rest of the proof into two cases according to the choice of $\Omega_{k}$. In case (V1), since $\Omega_{k}\left(\by^{k}\right)$, for any $k \in \nn$, is quasi Lipschitz (see Assumption \ref{A:Outer2}) the result follows from Remark \ref{R:BounQL} due to the fact that $\Seq{\by}{k}$ is bounded. Second, in case (V2), we denote $\bw^{k} := \by^{k} - \eta_{k}\nabla \sigma\left(\by^{k}\right)$. Now, using the definition of the gradient mapping (see \eqref{GradMap}) and \eqref{V2:Outer} we get that
				\begin{equation} \label{L:Basic:1}
					\bx^{k + 1} = T_{\eta_{k}}^{\sigma , \psi}\left(\by^{k}\right) = \prox_{\eta_{k}\psi}\left(\bw^{k}\right).
				\end{equation}
				Hence, from \eqref{U:Outer}, we obtain from the triangle inequality that
				\begin{align}
    					\norm{\Omega_{k}\left(\by^{k}\right)} & = \frac{1}{\eta_{k}} \norm{\bx^{k + 1} - \by^{k}} \nonumber \\
    					& \leq \frac{1}{\eta_{k}}\norm{\bx^{k + 1} - \bw^{k}} + \frac{1}{\eta_{k}}\norm{\bw^{k} - \by^{k}} \nonumber \\
    					& = \frac{1}{\eta_{k}}\norm{\prox_{\eta_{k}\psi}\left(\bw^{k}\right) - \bw^{k}} + \norm{\nabla \sigma\left(\by^{k}\right)}, \label{L:Basic:2}
    				\end{align}
    				where the last equality follows from \eqref{L:Basic:1} and the definition of $\bw^{k}$. Now, we show that each term in \eqref{L:Basic:2} is bounded. First, from Proposition \ref{P:SubGradProperty}, we get that 
				\begin{equation*}
					\frac{1}{\eta_{k}}\left(\bw^{k} - \prox_{\eta_{k}\psi}\left(\bw^{k}\right)\right) \in \partial \psi\left(\prox_{\eta_{k}\psi}\left(\bw^{k}\right)\right) = \partial \psi\left(\bx^{k + 1}\right),
				\end{equation*}
				where the last equality follows again from \eqref{L:Basic:1}. Second, since $\psi$ is with a full domain (see Assumption \ref{A:Outer2} (C2)) it has bounded sub-gradients on bounded subsets (see \cite[Theorem 3.16, p. 42]{B2017-B}). Thus, from the fact that $\Seq{\bx}{k}$ is bounded as assumed, we have that the first term in \eqref{L:Basic:2} is bounded. The boundedness of the second term follows from the facts that $\nabla \sigma$ is Lipschitz continuous (see Assumption \ref{A:Outer2} (C2)) and that $\Seq{\by}{k}$ is bounded as was proved above.
		\end{itemize}\vspace{-0.2in}
	\end{proof}
	We will prove now the boundedness of the sequence $\left\{ \norm{\bx^{k} - \bx'} \right\}_{k \in \nn}$, which will be used also below in proving the rate of convergence results.
	\begin{theorem}[Boundedness result] \label{T:Bounded}
		Let $\Seq{\bx}{k}$ and $\Seq{\by}{k}$ be the two sequences generated by the Bi-SG Algorithm. Then, for any $\bx' \in X'$, there exist positive constants $D_{1}$ and $D_{2}$, such that $\norm{\bx^{k} - \bx'} \leq D_{1}$ and $\norm{\Omega_{k}\left(\by^{k}\right)} \leq D_{2}$ for all $k \geq 0$.
	\end{theorem}
	\begin{proof}
    		Set $k \in \nn$ and take $\bx' \in X'$. It should be noted that according to Lemma \ref{L:Basic}(iii), we focus on finding $D_{1} > 0$ such that $\norm{\bx^{k} - \bx'} \leq D_{1}$, because this will immediately implies the existence of $D_{2}$ such that $\norm{\Omega_{k}\left(\by^{k}\right)} \leq D_{2}$. To this end, we split now the proof according to the two choices of the mapping $\Omega_{k}$.
		\begin{itemize}
    			\item[$\rm{(i)}$] In this case, $\Omega_{k}\left(\by^{k}\right) \in \partial \omega\left(\by^{k}\right)$. Since $\partial \omega$ is quasi Lipschitz (see Assumption \ref{A:Outer2} (C1)) there exists $d_{1}, d_{2} > 0$ such that for any $\bx \in \real^{n}$, we have that $\norm{\partial \omega\left(\bx\right)} \leq \max\left\{d_{1}, d_{2}\norm{\bx}\right\}$. We denote $R := \max \left\{ L , d_{1}/(2d_{2}) , \norm{\bx'} \right\}$ where $L := \sup \{ \norm{\bx - \bx'} : \, \omega\left(\bx\right) \leq \omega\left(\bx'\right) \}$ is finite since $\omega$ is coercive (see Assumption \ref{A:Outer}). We consider two cases:    		
    				\begin{itemize}
    					\item[(a)] If $\norm{\by^{k} - \bx'} \leq R$, then from \eqref{U:Outer} we have that $\bx^{k + 1} = \by^{k} - \eta_{k}\Omega_{k}\left(\by^{k}\right)$ and thus
    						\begin{equation*}
    							\norm{\bx^{k + 1} - \bx'} \leq \norm{\by^{k} - \bx'} + \eta_{k}\norm{\Omega_{k}\left(\by^{k}\right)} \leq R + \norm{\Omega_{k}\left(\by^{k}\right)},
    						\end{equation*}
    						where the last inequality follows from the fact that $\eta_{k} \leq 1$. In addition, since $\omega$ is with a full domain (see Assumption \ref{A:Outer}) it has bounded sub-gradients on bounded subsets (see \cite[Theorem 3.16]{B2017-B}) and thus $\norm{\Omega_{k}\left(\by^{k}\right)}$ is bounded since $\Omega_{k}\left(\by^{k}\right) \in \partial \omega\left(\by^{k}\right)$. Therefore, in this case, there exists $C_{1} > 0$ such that $\norm{\bx^{k + 1} - \bx'} \leq C_{1}$.
    					\item[(b)] If $\norm{\by^{k} - \bx'} > R$, then from \eqref{U:Outer} we have that    
    						\begin{align}
         					\norm{\bx^{k + 1} - \bx'}^{2} & = \norm{\by^{k} - \eta_{k}\Omega_{k}\left(\by^{k}\right) - \bx'}^{2} \nonumber \\
         					& = \norm{\by^{k} - \bx'}^{2} - 2\eta_{k}\act{\by^{k} - \bx' , \Omega_{k}\left(\by^{k}\right)} + \eta_{k}^{2}\norm{\Omega_{k}\left(\by^{k}\right)}^{2} \nonumber \\
      						& \leq \norm{\by^{k} - \bx'}^{2} - 2\eta_{k}\left(\omega\left(\by^{k}\right) - \omega\left(\bx'\right)\right) + \eta_{k}^{2}\norm{\Omega_{k}\left(\by^{k}\right)}^{2}, \label{T:Bounded:1}
						\end{align}
      					where the last inequality follows from the sub-gradient inequality of the convex function $\omega\left(\cdot\right)$ and the fact that $\Omega_{k}\left(\by^{k}\right) \in \partial \omega\left(\by^{k}\right)$. Now, since $\norm{\by^{k} - \bx'} > R \geq L$ and by the definition of $L$ we get that $\omega\left(\bx'\right) < \omega\left(\by^{k}\right)$. Indeed, if by contradiction $\omega\left(\by^{k}\right) \leq \omega\left(\bx'\right)$, then, from the definition of $L$, it follows that $L \geq \norm{\by^{k} - \bx'}$. However, using the case assumption that $\norm{\by^{k} - \bx'} > R$ and the definition of $R$ we have that $\norm{\by^{k} - \bx'}> R \geq L$, which contradicts the fact that $L \geq \norm{\by^{k} - \bx'}$. Thus, $\omega\left(\bx'\right) < \omega\left(\by^{k}\right)$, and we have that
    						\begin{align}
         					\norm{\bx^{k + 1} - \bx'}^{2} & \leq \norm{\by^{k} - \bx'}^{2} + \eta_{k}^{2}\norm{\Omega_{k}\left(\by^{k}\right)}^{2} \nonumber \\
        						& \leq \norm{\bx^{k} - \bx'}^{2} + \eta_{k}^{2}\norm{\Omega_{k}\left(\by^{k}\right)}^{2}, \label{T:Bounded:1}
						\end{align}    						
						where the last inequality follows from Lemma \ref{L:Basic}(ii). Since $\Omega_{k}$ is quasi Lipschitz with constants $\left(d_{1} , d_{2}\right)$ by our assumption in this case (see Assumption \ref{A:Outer2} (C1)), we obtain from Definition \ref{D:QuasiL} that
						\begin{align}
     						\norm{\Omega_{k}\left(\by^{k}\right)} & \leq \max \left\{ d_{1} , d_{2}\norm{\by^{k}} \right\} \nonumber \\
     						& \leq \max \left\{ d_{1} , d_{2}\left(\norm{\by^{k} - \bx'} + \norm{\bx'}\right) \right\} \nonumber \\
     						& \leq \max\{d_{1}, 2d_{2}\norm{\by^{k} - \bx'} \} \nonumber \\
     						& \leq 2d_{2}\norm{\by^{k} - \bx'} \nonumber \\
     						& \leq 2d_{2}\norm{\bx^{k} - \bx'}, \label{T:Bounded:2}
						\end{align}
						where the third and fourth inequalities follow from the facts that $\norm{\by^{k} - \bx'} > R \geq \norm{\bx'}$ and $R \geq d_{1}/\left(2d_{2}\right)$, while the last inequality follows from Lemma \ref{L:Basic}(ii). Thus, by combining \eqref{T:Bounded:1} and \eqref{T:Bounded:2}, we obtain 
						\begin{equation*}
    							\norm{\bx^{k + 1} - \bx'}^{2} \leq \left(1 + 4d_{2}^{2}\eta_{k}^{2}\right)\norm{\bx^{k} - \bx'}^{2}.
						\end{equation*}
				\end{itemize}
				Therefore, if we combine the two cases together, we have for all $k \geq 0$ that
				\begin{equation*}
    					\norm{\bx^{k + 1} - \bx'}^{2} \leq \max \left\{ C_{1}^{2} , \left(1 + 4d_{2}^{2}\eta_{k}^{2}\right)\norm{\bx^{k} - \bx'}^{2} \right\},
				\end{equation*}
				and by a simple induction we get that			
				\begin{equation} \label{T:Bounded:3}
    					\norm{\bx^{k} - \bx'}^{2} \leq \prod_{s = 0}^{k - 1} \left(1 + 4d_{2}^{2}\eta_{s}^{2}\right) \max \left\{ C_{1}^{2} , \norm{\bx^{0} - \bx'}^{2} \right\}.
				\end{equation}
				Now, we will show that the product term is bounded from above. First, we have
				\begin{equation*}
					\log\left(\prod_{s = 0}^{k - 1} \left(1 + 4d_{2}^{2}\eta_{s}^{2}\right)\right) = \sum_{s = 0}^{k - 1} \log\left(1 + 4d_{2}^{2}\eta_{s}^{2}\right) \leq \sum_{s = 0}^{k - 1} 4d_{2}^{2}\eta_{s}^{2} \nonumber
				\end{equation*}
				where the inequality follows from the fact that $\log\left(1 + t\right) \leq t$ for any $t > 0$. In addition, since $\eta_{s} = c\left(1 + s\right)^{-\alpha}$ for $\alpha \in \left(1/2 , 1 \right)$ and $c \leq 1$ we get that
				\begin{equation*}
    					\sum_{s = 0}^{k - 1} \eta_{s}^{2} = c^{2}\sum_{s = 0}^{k - 1} \left(s + 1\right)^{-2\alpha} \leq \sum_{s = 1}^{k} s^{-2\alpha} = 1 + \sum_{s = 2}^{k} s^{-2\alpha} \leq 1 + \frac{1^{1 - 2\alpha} - k^{1 - 2\alpha}}{2\alpha - 1} \leq \frac{2\alpha}{2\alpha - 1},					
				\end{equation*}
				where the second inequality follows from Lemma \ref{L:Sum} with $n_{1} = 2$, $n_{2} = k$ and $r = 2\alpha > 1$. Therefore, we immediately deduce from \eqref{T:Bounded:3} that there exists $D_{1} > 0$ such that $\norm{\bx^{k} - \bx'} \leq D_{1}$ for all $k \geq 0$.
    			\item[$\rm{(ii)}$] In this case, from (V2) we have that $\Omega_{k} = G_{\eta_{k}}^{\sigma , \psi}$. We split the proof into two cases. If $\omega\left(\bx^{k + 1}\right) \leq \omega\left(\bx'\right)$, then since $\omega$ is coercive it has bounded level-sets (see \cite[Proposition 11.11]{BC2011-B}) and therefore there exists $C_{1} > 0$ such that $\norm{\bx^{k + 1} - \bx'} \leq C_{1}$. Otherwise, we have that $\omega\left(\bx^{k + 1}\right) > \omega\left(\bx'\right)$. Thus, using Proposition \ref{P:ProximalInequality} with $s = \sigma$, $h = \psi$, $t = \eta_{k}$, $\bx = \by^{k}$ and $\by = \bx'$, we have from \eqref{P:ProximalInequality:2} that
   				\begin{equation*}
      				\norm{\bx^{k + 1} - \bx'}^{2} \leq \norm{\by^{k} - \bx'}^{2} - 2\eta_{k}\left(\omega\left(\bx^{k + 1}\right) - \omega\left(\bx'\right)\right) \leq \norm{\by^{k} - \bx'}^{2} \leq \norm{\bx^{k} - \bx'}^{2}
    				\end{equation*}
				where the last inequality follows from Lemma \ref{L:Basic}(ii). Hence, by combining both cases, we have for all $k \geq 0$ that $\norm{\bx^{k + 1} - \bx'} \leq \max \left\{ C_{1} , \norm{\bx^{k} - \bx'} \right\}$. Now, by using a simple induction, we obtain the desired result.
		\end{itemize}\vspace{-0.35in}
	\end{proof}
	\begin{remark} \label{R:BoundedSub}
		It should be noted that in case (V2), using the arguments mentioned in the proofs of Theorem \ref{T:Bounded} and Lemma \ref{L:Basic}(iii), it is also easy to show that any $\xi^{k} \in \partial \omega\left(\by^{k}\right)$ are bounded for all $k \geq 0$. For simplicity, we will assume that the upper bound is also $D_{2}$, that is, $\norm{\xi^{k}} \leq D_{2}$ for all $k \geq 0$.
	\end{remark}

\subsection{Rate of Convergence of Bi-SG - Convex Case} \label{SSec:Rate}
	Now, we are ready to prove the rate of convergence results of the Bi-SG algorithm in the case where the outer objective function $\omega\left(\cdot\right)$ is convex. The analysis below will be valid for all $\alpha \in \left(1/2 , 1\right]$, however in the case where $\alpha  = 1$ we will only have the inner rate. It should be noted that our results are also valid for the case $\alpha =1/2$, however the boundedness result above is not valid as described in the following remark.
	\begin{remark}
		It should be noted $\alpha > 1/2$ is actually only used in the proof of Theorem 1, meaning in proving the boundedness property. Moreover, it is only needed for the first version of Bi-SG (see the last inequality in the proof of Theorem 1, case (i)). 
	\end{remark}
	Throughout this part we will use the following two constants:
		\begin{itemize}
    			\item[$\rm{(i)}$] $D \equiv D_{1} + D_{2}$, where $D_{1}$ and $D_{2}$ are given in Theorem \ref{T:Bounded}.
    			\item[$\rm{(ii)}$] 
    			\begin{equation*}
    				H \equiv \begin{cases}
    					D^{2}\bar{L} / \left(1 - \alpha\right), & \alpha \in \left(1/2 , 1\right), \\
    					D^{2}\bar{L}, & \alpha = 1, 
				\end{cases}
			\end{equation*}    			
    			where $\bar{L}$ is given in Proposition \ref{P:BackT}. 
		\end{itemize}
\medskip

	We begin with a rate of convergence result in terms of the inner objective function $\varphi\left(\cdot\right)$. To this end, we first prove the following technical result.
	\begin{lemma} \label{L:RateTech}
		Let $\Seq{\bx}{k}$ and $\Seq{\by}{k}$ be the two sequences generated by the Bi-SG Algorithm. For any $\bx' \in X'$ and $k \geq 1$, we have
		\begin{equation*}
			\sum_{j = 1}^{k} \left(\varphi\left(\by^{j}\right) - \varphi\left(\bx'\right)\right) \leq H						\begin{cases}
    					k^{1-\alpha}, & \alpha \in \left(1/2 , 1\right), \\
    					\ln\left(k\right) + 1, & \alpha = 1. 
				\end{cases}
		\end{equation*}
	\end{lemma}	
	\begin{proof} 
		First, for any $j \geq 0$, we use Proposition \ref{P:ProximalInequality} with $s = f$, $h = g$, $\bx = \bx^{j + 1}$, $\by = \bx'$ and $t = t_{j} = 1 / L_{j}$ (see \eqref{P:ProximalInequality:2}), to obtain that
    		\begin{equation} \label{L:RateTech:1}
        		\varphi\left(\by^{j + 1}\right) - \varphi\left(\bx'\right) \leq \frac{L_{j}}{2}\left(\norm{\bx^{j + 1} - \bx'}^{2} - \norm{\by^{j + 1} - \bx'}^{2} \right).
        	\end{equation}   
        	On the other hand, from \eqref{U:Outer}, we have that
        	\begin{align}
        		\norm{\bx^{j + 1} - \bx'}^{2} & = \norm{\by^{j} - \eta_{j}\Omega_{j}\left(\by^{j}\right) - \bx'}^{2} \nonumber \\
        		& = \norm{\by^{j} - \bx'}^{2} - 2\eta_{j}\act{\Omega_{j}\left(\by^{j}\right), \by^{j} - \bx'} + \eta_{j}^{2}\norm{\Omega_{j}\left(\by^{j}\right)}^{2} \nonumber \\
        		& \leq \norm{\by^{j} - \bx'}^{2} + 2\eta_{j}D_{1}D_{2} + \eta_{j}D_{2}^{2} \nonumber \\
        		& \leq \norm{\by^{j} - \bx'}^{2} + \eta_{j}D^{2}, \label{L:RateTech:2}
		\end{align}        	     		
		where the first inequality follows from Theorem \ref{T:Bounded} and the fact that $\eta_{j} \leq 1$, the last inequality follows from the fact that $D = D_{1} + D_{2} > 0$. Combining \eqref{L:RateTech:1} with \eqref{L:RateTech:2} and using the fact that $L_{j} \leq \bar{L}$ (see Proposition \ref{P:BackT}), yields
        	\begin{align}
        		\varphi\left(\by^{j + 1}\right) - \varphi\left(\bx'\right) & \leq \frac{L_{j}}{2} \left(\norm{\bx^{j + 1} - \bx'}^{2} - \norm{\by^{j + 1} - \bx'}^{2} \right) \nonumber \\
        		& \leq \frac{\bar{L}}{2}\left(\norm{\bx^{j + 1} - \bx'}^{2} - \norm{\by^{j + 1} - \bx'}^{2} \right) \nonumber \\
        		& \leq \frac{\bar{L}}{2}\left(\norm{\by^{j} - \bx'}^{2} - \norm{\by^{j + 1} - \bx'}^{2} \right) + \frac{\eta_{j}}{2}D^{2}\bar{L}. \label{L:RateTech:3} 
    		\end{align}
		Summing \eqref{L:RateTech:3}, for all $0 \leq j \leq k - 1$, gives that
		\begin{align}
    			\sum_{j = 1}^{k} \left(\varphi\left(\by^{j}\right) - \varphi\left(\bx'\right)\right) & = \sum_{j = 0}^{k - 1} \left(\varphi\left(\by^{j + 1}\right) - \varphi\left(\bx'\right)\right) \nonumber \\
    			& \leq \frac{\bar{L}}{2}\left(\norm{\by^{0} - \bx'}^{2} - \norm{\by^{k} - \bx'}^{2} \right) + \frac{D^{2}\bar{L}}{2}\sum_{j = 0}^{k - 1} \eta_{j} \nonumber \\
    			& \leq \frac{\bar{L}}{2}\norm{\by^{0} - \bx'}^{2} + \frac{D^{2}\bar{L}}{2}\sum_{j = 0}^{k - 1} \eta_{j} \nonumber \\
    			& \leq \frac{H}{2} + \frac{D^{2}\bar{L}}{2}\sum_{j = 0}^{k - 1} \eta_{j}, \label{L:RateTech:4}		
		\end{align}    		
		where the last inequality follows from Lemma \ref{L:Basic}(ii) and Theorem \ref{T:Bounded} since $\norm{\by^{0} - \bx'}^{2} \leq \norm{\bx^{0} - \bx'}^{2} \leq D_{1}^{2} \leq D^{2}$ and the fact that $D^{2}\bar{L} \leq H$ (for all $1/2 < \alpha \leq 1$). We split now the proof into two cases. If $1/2 < \alpha < 1$, using Lemma \ref{L:Sum} with $n_{1} = 1$, $n_{2} = k$ and $r = \alpha$, yields
		\begin{equation} \label{L:RateTech:5}
    			\sum_{j = 0}^{k - 1} \eta_{j} = c\sum_{j = 0}^{k - 1} \left(j + 1\right)^{-\alpha} \leq \sum_{j = 1}^{k} j^{-\alpha} \leq \frac{k^{1 - \alpha} - 0^{1 - \alpha}}{1 - \alpha} = \frac{k^{1 - \alpha}}{1 - \alpha},
		\end{equation}
		where the first inequality follows from the fact that $c \leq 1$. By combining \eqref{L:RateTech:4} and \eqref{L:RateTech:5} we obtain that (recall that in this case $H = D^{2}\bar{L}/\left(1 - \alpha\right)$)
		\begin{equation*}
			\sum_{j = 1}^{k} \left(\varphi\left(\by^{j}\right) - \varphi\left(\bx'\right)\right) \leq \frac{1}{2}H + \frac{1}{2}Hk^{1 - \alpha} \leq Hk^{1 - \alpha},
		\end{equation*}
		where the last inequality follows from the fact that $k^{1 - \alpha} \geq 1$ since $\alpha < 1$ and $k \geq 1$. When $\alpha = 1$, we use classical properties of the harmonic series, to obtain that
		\begin{equation*}
    			\sum_{j = 0}^{k - 1} \eta_{j} = c\sum_{j = 0}^{k - 1} \left(j + 1\right)^{-1} \leq \sum_{j = 1}^{k} j^{-1} \leq \ln\left(k\right) + 1,
		\end{equation*}
		which combined with \eqref{L:RateTech:4}, proves the desired result.
	\end{proof}
	Now, we are ready to prove a rate of convergence result for the inner problem (see \eqref{Prob:IP}). This result will be useful in proving the main result of this section, which is the rate of convergence of Bi-SG in terms of both the inner and the outer objective functions.
	\begin{theorem}[An inner rate of convergence of Bi-SG] \label{T:InnerRate}
		Let $\Seq{\bx}{k}$ and $\Seq{\by}{k}$ be the two sequences generated by the Bi-SG Algorithm. For any $\bx' \in X'$ and $k \geq 1$, we have that
		\begin{equation} \label{T:InnerRate:0}
			\varphi\left(\by^{k}\right) - \varphi\left(\bx'\right) \leq 2H\begin{cases}
    					1/k^{\alpha}, & \alpha \in \left(1/2 , 1\right), \\
    					\left(\ln\left(k\right) + 1\right)/k, & \alpha = 1. 
				\end{cases}
		\end{equation}
	\end{theorem}
	\begin{proof}
		Fix $k \geq 1$ and for simplicity we denote $G_{k} = G_{L_{k}}^{f , g}\left(\bx^{k}\right)$ and $v_{k} = \varphi\left(\by^{k}\right) - \varphi\left(\bx'\right)$. In addition, using Proposition \ref{P:ProximalInequality} with $s = f$, $h = g$, $t = t_{k + 1} = 1 / L_{k + 1}$, $\bx = \bx^{k + 1}$ and $\by = \by^{k}$ (see \eqref{P:ProximalInequality:3}), we obtain
		\begin{align}
    			v_{k + 1} - v_{k} & = \varphi\left(\by^{k + 1}\right) - \varphi\left(\by^{k}\right) \nonumber \\
    			& \leq \act{G_{k + 1} , \bx^{k + 1} - \by^{k}} - \frac{1}{2L_{k + 1}}\norm{G_{k + 1}}^{2} \nonumber \\
    			& = -\eta_{k}\act{G_{k + 1} , \Omega_{k}\left(\by^{k}\right)} - \frac{1}{2L_{k + 1}}\norm{G_{k + 1}}^{2} \nonumber \\
    			& \leq \norm{G_{k + 1}}\left(\eta_{k}D_{2} - \frac{1}{2\bar{L}}\norm{G_{k + 1}}\right), \label{T:InnerRate:1}
		\end{align}
		where the second equality follows from the the definition of $\bx^{k + 1}$ (see \eqref{U:Outer}) and the last inequality follows from the facts that $L_{k + 1} \leq \bar{L}$ (see Proposition \ref{P:BackT}) and that $\norm{\Omega_{k}\left(\by^{k}\right)} \leq D_{2}$ as was proved in Theorem \ref{T:Bounded}. Furthermore, applying again Proposition \ref{P:ProximalInequality} with $s = f$, $h = g$, $t = t_{k + 1} = 1 / L_{k + 1}$, $\bx = \bx^{k + 1}$ and $\by = \bx'$ (see \eqref{P:ProximalInequality:3}), gives		
		\begin{equation} \label{T:InnerRate:2}
    			v_{k + 1} = \varphi\left(\by^{k + 1}\right) - \varphi\left(\bx'\right) \leq \act{G_{k + 1} , \bx^{k + 1} - \bx'} - \frac{1}{2L_{k + 1}}\norm{G_{k + 1}}^{2} \leq D_{1}\norm{G_{k + 1}},
		\end{equation}
		where the last inequality follows from the Cauchy-Schwartz inequality and the definition of $D_{1}$ (see Theorem \ref{T:Bounded}). Thus, by combining \eqref{T:InnerRate:1} and \eqref{T:InnerRate:2}, we immediately get
		\begin{equation*}
    			v_{k + 1} - v_{k} \leq \norm{G_{k + 1}}\left(\eta_{k}D_{2} - \frac{1}{2\bar{L}D_{1}}v_{k + 1}\right).
		\end{equation*}
		Hence, we get that $v_{k + 1} \geq v_{k}$ can be happening only if 
		\begin{equation} \label{T:InnerRate:3}
    			v_{k + 1} \leq 2D_{1}D_{2}\bar{L}\eta_{k} = 2D_{1}D_{2}\bar{L}c\left(k + 1\right)^{-\alpha} \leq Hk^{-\alpha},
		\end{equation}
		where the equality follows from the definition of $\eta_{k}$ and the second inequality follows from the definition of $H$ (since $2D_{1}D_{2} \leq D^{2}$) and the fact that $c \leq 1$.
\medskip

		Now, from Lemma \ref{L:RateTech}, after dividing both sides by $k$, we get that 
$\left(1/k\right)\sum_{j = 1}^{k} v_{j} \leq Hr_{k}$, where $r_{k} = k^{-\alpha}$ (when $1/2 < \alpha < 1$) and $r_{k} = \left(\ln\left(k\right) + 1\right)/k$ (when $\alpha = 1$). Thus, since $v_{j} \geq 0$ for any $j \in \nn$, we can easily deduce that there exist at least $\lceil k/2 \rceil$ indices $j \in \left\{ 1 , 2 , \ldots , k\right\}$ which satisfy $v_{j} \leq 2Hr_{k}$. Thus, it is obvious that the last such index, which we denote by $j^{\ast}$, satisfies that $j^{\ast} \geq k/2$.
\medskip

		Now, let us split the proof into two cases. First, we assume that there exists an index $j \in \left[j^{\ast} , k - 1\right]$ for which $v_{j + 1} > v_{j}$, and we denote the last such index by $\tilde{j}$. By the definition of $\tilde{j}$, we have that
		\begin{equation}
	    		v_{k} \leq v_{k - 1} \leq \cdots \leq v_{\tilde{j} + 2} \leq v_{\tilde{j} + 1}.
		\end{equation}
		In addition, since $v_{\tilde{j} + 1} > v_{\tilde{j}}$ and $k/2 \leq j^{*} \leq \tilde{j}$, it follows from \eqref{T:InnerRate:3} that
		\begin{equation*}
    			v_{k} \leq v_{\tilde{j} + 1} \leq H\tilde{j}^{-\alpha} \leq 2Hk^{-\alpha} \leq 2Hr_{k},
		\end{equation*}
		where the last inequality follows from the fact that $1 < \ln\left(k\right) + 1$, since $k \geq 1$. 
\medskip
		
		Second, there is no such an index $j$ for which $v_{j + 1} > v_{j}$. Thus, from the definition of $j^{\ast}$, we have that
		\begin{equation*}
    			v_{k} \leq v_{j^{*}} \leq 2Hr_{k},
		\end{equation*}
		which proves the desired result.
	\end{proof}
	The following technical result will be essential in obtaining the convergence rate in terms of the outer objective function (see problem \eqref{Prob:OP}). It should be noted this result is valid when $1/2 < \alpha < 1$.
	\begin{lemma} \label{L:OuterTech}
		Let $\Seq{\bx}{k}$ and $\Seq{\by}{k}$ be the two sequences generated by the Bi-SG Algorithm. For any $\bx' \in X'$ and $k \geq 1$, we have
		\begin{equation*}
			u_{k} \leq \frac{D^{2}}{ck^{1 - \alpha}},
		\end{equation*}
		where $u_{k} = \min \left\{ \act{\Omega_{j}\left(\by^{j}\right), \by^{j} - \bx'} : \, k \leq j \leq 2k \right\}$.	
	\end{lemma}	
	\begin{proof} 
		Let $k \geq 1$. If $u_{k} \leq 0$ then the result is obvious. Therefore, from now on we assume that $u_{k} \geq 0$. For all $j \geq 0$, we use the fact that $\bx^{j + 1} = \by^{j} - \eta_{j}\Omega_{j}\left(\by^{j}\right)$ to derive that
    		\begin{align*}
         	\norm{\bx^{j + 1} - \bx'}^{2} & = \norm{\by^{j} - \bx'}^{2} - 2\eta_{j}\act{\Omega_{j} \left(\by^{j}\right) , \by^{j} - \bx'} + \eta_{j}^{2}\norm{\Omega_{j}\left(\by^{j}\right)}^{2} \\
         	& \leq \norm{\bx^{j} - \bx'}^{2} - 2\eta_{j}\act{\Omega_{j} \left(\by^{j}\right) , \by^{j} - \bx'} + \eta_{j}^{2}D_{2}^{2},
         \end{align*}
		where the inequality follows from Lemma \ref{L:Basic}(ii) and Theorem \ref{T:Bounded}. Hence	
		\begin{equation*}
     		2\eta_{j}\act{\Omega_{j}\left(\by^{j}\right) , \by^{j} - \bx'} \leq \norm{\bx^{j} - \bx'}^{2} - \norm{\bx^{j + 1} - \bx'}^{2} + \eta^{2}_{j}D_{2}^{2}.
		\end{equation*}
		Summing it, for all $j = k , k + 1 , \ldots , 2k - 1$, yields
	   	\begin{equation*}
    			2\sum_{j = k}^{2k - 1} \eta_{j}\act{\Omega_{j}\left(\by^{j}\right) , \by^{j} - \bx'} \leq \norm{\bx^{k} - \bx'}^{2} - \norm{\bx^{2k} - \bx'}^{2} + D_{2}^{2}\sum_{j = k}^{2k - 1} \eta_{j}^{2}\leq D_{1}^{2} + D_{2}^{2}\sum_{j = k}^{2k - 1} \eta_{j}^{2},
    		\end{equation*}
    		where the last inequality follows from Theorem \ref{T:Bounded}. Now, since $\eta_{k} > \eta_{j}$ for all $k \leq j \leq 2k - 1$, we obtain that
    		\begin{equation*}
    			\sum_{j = k}^{2k - 1} \eta_{j}^{2} \leq \left(2k - 1 - k + 1\right)\eta_{k}^{2} = k\eta_{k}^{2} \leq 1,
    		\end{equation*}
    		where the last inequality follows from the definition of $\eta_{k}$ since $c \leq 1$ and $2\alpha \geq 1$. This implies that (recall $D = D_{1} + D_{2}$)
	   	\begin{equation} \label{L:OuterTech:1}
    			2\sum_{j = k}^{2k - 1} \eta_{j}\act{\Omega_{j}\left(\by^{j}\right) , \by^{j} - \bx'} \leq D_{1}^{2} + D_{2}^{2} \leq D^{2}.
    		\end{equation}
    		Now, from the definition of $u_{k}$ and since $u_{k} \geq 0$, we get that
		\begin{equation*}
			2\sum_{j = k}^{2k - 1} \eta_{j}\act{\Omega_{j}\left(\by^{j}\right) , \by^{j} - \bx'} \geq 2u_{k}\sum_{j = k}^{2k - 1} \eta_{j} \geq 2u_{k}\left(2k - 1 - k + 1\right)\eta_{2k - 1} = 2u_{k}k\eta_{2k - 1},
		\end{equation*}		
		where the last inequality follows from the fact that $\eta_{j} \geq \eta_{2k - 1}$ for all $k \leq j \leq 2k - 1$. Since $\alpha \leq 1$ we obtain that 2$k\eta_{2k - 1} = c\left(2k\right)^{1 - \alpha} \geq ck^{1 - \alpha}$, and thus combined with \eqref{L:OuterTech:1} yields that
	   	\begin{equation*}
    			cu_{k}k^{1 - \alpha}\leq D^{2},
    		\end{equation*}
		which proves the result. 
	\end{proof}
	Now, we are ready to prove the main result, which is valid when $1/2 < \alpha < 1$.
	\begin{theorem}[Rate of convergence] \label{T:SimuRate}
		Let $\Seq{\bx}{k}$ and $\Seq{\by}{k}$ be the two sequences generated by the Bi-SG Algorithm. For any $\bx' \in X'$ and $k \geq 1$, we have
		\begin{equation*}
    			\varphi\left(\by^{k_{best}}\right) - \varphi\left(\bx'\right) \leq \frac{2H}{k^{\alpha}} \qquad \text{and} \qquad \omega\left(\by^{k_{best}}\right) - \omega\left(\bx'\right) \leq \frac{2D^{2}}{ck^{1 - \alpha}},
		\end{equation*}
		where $k_{best} = \argmin \left\{ \omega\left(\by^{j}\right) : \, k \leq j \leq 2k \right\}$. 
	\end{theorem}		
	\begin{proof}
     	The rate in terms of the inner objective function $\varphi\left(\cdot\right)$ follows, directly, from Theorem \ref{T:InnerRate} using the fact that $k_{best} \geq k$.
\medskip
	
		In order to prove the rate in terms of the outer objective function $\omega\left(\cdot\right)$, we will use Lemma \ref{L:OuterTech} and for simplicity we define $\bar{k} = \argmin \left\{ \act{\Omega_{j}\left(\by^{j}\right), \by^{j} - \bx'} : \, k \leq j \leq 2k \right\}$. We split now the proof according to the choices of the mapping $\Omega_{k}$, $k \in \nn$. In case (V1), where $\Omega_{k}\left(\by^{k}\right) \in \partial \omega\left(\by^{k}\right)$, we apply the sub-gradient inequality on the convex function $\omega\left(\cdot\right)$ to obtain 
		\begin{equation*}
        		\omega\left(\by^{k_{best}}\right) - \omega\left(\bx'\right) \leq \omega\left(\by^{\bar{k}}\right) - \omega\left(\bx'\right) \leq \act{\Omega_{\bar{k}}\left(\by^{\bar{k}}\right) , \by^{\bar{k}} - \bx'} \leq \frac{D^{2}}{ck^{1 - \alpha}}.
	    \end{equation*}
    		where the first inequality follows from the fact that $\omega\left(\by^{k_{best}}\right) \leq \omega \left(\by^{\bar{k}}\right)$ and the last inequality follows from Lemma \ref{L:OuterTech}.
\medskip

		Now, let us consider case (V2). First, by applying Proposition \ref{P:ProximalInequality} with $s = \sigma$, $h = \psi$, $t = \eta_{\bar{k}}$, $\bx = \by^{\bar{k}}$ and $\by = \bx'$ (see \eqref{P:ProximalInequality:3}), together with Lemma \ref{L:OuterTech}, we have that
		\begin{equation}
    			\omega\left(\bx^{\bar{k} + 1}\right) - \omega\left(\bx'\right) \leq \act{\Omega_{\Bar{k}}\left(\by^{\bar{k}}\right) , \by^{\bar{k}} - \bx'} - \frac{\eta_{\bar{k}}}{2}\norm{\Omega_{\bar{k}} \left(\by^{\bar{k}}\right)}^{2} \leq \frac{D^{2}}{ck^{1 - \alpha}}.
		\end{equation}
		In addition, let $\xi \in \partial \omega\left(\by^{\bar{k}}\right)$, and applying the sub-gradient inequality on the convex function $\omega\left(\cdot\right)$ to obtain 
		\begin{equation*}
    			\omega\left(\by^{k_{best}}\right) - \omega\left(\bx^{\bar{k} + 1}\right) \leq \omega\left(\by^{\bar{k}}\right) - \omega\left(\bx^{\bar{k} + 1}\right) \leq \act{\xi , \by^{\bar{k}} - \bx^{\bar{k} + 1}} = \eta_{\bar{k}}\act{\xi , \Omega_{\bar{k}}\left(\by^{\bar{k}}\right)} \leq \eta_{\bar{k}}D^{2},
    		\end{equation*}
    		where the equality follows from \eqref{U:Outer} and the last inequality follows from Theorem \ref{T:Bounded} and Remark \ref{R:BoundedSub} since $D_{2} < D$. Moreover, since $\bar{k} \geq k > k - 1$ we have that $\eta_{\bar{k}} \leq \eta_{k - 1} = ck^{-\alpha} \leq 1/k^{1 - \alpha}$ since $c \leq 1$ and $\alpha > 1/2$. Combining all these facts yields
		\begin{equation*}
    			\omega\left(\by^{k_{best}}\right) - \omega\left(\bx'\right) =  \omega\left(\by^{k_{best}}\right) - \omega\left(\bx^{\bar{k} + 1}\right)  + \omega\left(\bx^{\bar{k} + 1}\right) - \omega\left(\bx'\right) \leq \frac{D^{2}}{ck^{1 - \alpha}} + \frac{D^{2}}{k^{1 - \alpha}},
    		\end{equation*}
    		which implies the desired result since $c \leq 1$.
	\end{proof}
	The result above, which concerns with the convergence rates of the sequence generated by Bi-SG, is limited to the case where $1/2 < \alpha < 1$. For $\alpha = 1$ we can prove the following weaker result concerning the rate of the outer problem \eqref{Prob:OP}.
	\begin{lemma} \label{L:OuterR1}
		Let $\Seq{\bx}{k}$ and $\Seq{\by}{k}$ be the two sequences generated by the Bi-SG Algorithm Version I (\ie $\Omega_{k}$ is set by rule (V1)) and $\alpha = 1$. For any $\bx' \in X'$ and $k \geq 1$, we have
		\begin{equation*}
			\min_{1 \leq j \leq k} \omega\left(\by^{j}\right) - \omega\left(\bx'\right) \leq \frac{D^{2}}{2c\left(\ln\left(k + 1\right) - 1\right)}.
		\end{equation*}
	\end{lemma}
	\begin{proof} 
		Let $k \geq 1$. For all $j \geq 1$, we use the fact that $\bx^{j + 1} = \by^{j} - \eta_{j}\Omega_{j}\left(\by^{j}\right)$ to derive that
    		\begin{align}
         	\norm{\bx^{j + 1} - \bx'}^{2} & = \norm{\by^{j} - \bx'}^{2} - 2\eta_{j}\act{\Omega_{j} \left(\by^{j}\right) , \by^{j} - \bx'} + \eta_{j}^{2}\norm{\Omega_{j}\left(\by^{j}\right)}^{2} \nonumber\\
         	& \leq \norm{\bx^{j} - \bx'}^{2} - 2\eta_{j}\act{\Omega_{j} \left(\by^{j}\right) , \by^{j} - \bx'} + \eta_{j}^{2}D_{2}^{2}, \label{L:OuterR1:1}
         \end{align}
		where the inequality follows from Theorem \ref{T:Bounded}. Since $\Omega_{k}\left(\by^{k}\right) \in \partial \omega\left(\by^{k}\right)$, we apply the sub-gradient inequality on the convex function $\omega\left(\cdot\right)$ to obtain 
		\begin{equation*}
        		2\eta_{j}\left(\omega\left(\by^{j}\right) - \omega\left(\bx'\right)\right) \leq 2\eta_{j}\act{\Omega_{j}\left(\by^{j}\right) , \by^{j} - \bx'} \leq \norm{\bx^{j} - \bx'}^{2} - \norm{\bx^{j + 1} - \bx'}^{2} + \eta^{2}_{j}D_{2}^{2},
	    \end{equation*}
		where the last inequality follows from \eqref{L:OuterR1:1}. Summing it, for all $j = 1 , 2 , \ldots , k$, yields
	   	\begin{equation*}
    			2\sum_{j = 1}^{k} \eta_{j}\left(\omega\left(\by^{j}\right) - \omega\left(\bx'\right)\right) \leq \norm{\bx^{1} - \bx'}^{2} - \norm{\bx^{k + 1} - \bx'}^{2} + D_{2}^{2}\sum_{j = 1}^{k} \eta_{j}^{2}\leq D_{1}^{2} + D_{2}^{2}\sum_{j = 1}^{k} \eta_{j}^{2},
    		\end{equation*}
    		where the last inequality follows from Theorem \ref{T:Bounded}. Now, from the definition of $\eta_{j}$ and since $c \leq 1$, we obtain that
    		\begin{equation*}
    			\sum_{j = 1}^{k} \eta_{j}^{2} \leq \sum_{j = 1}^{k} \left(j + 1\right)^{-2} = \sum_{j = 2}^{k + 1} j^{-2} \leq \frac{\left(k + 1\right)^{-1} - 1}{-1} \leq 1,
    		\end{equation*}
    		where the second inequality follows from Lemma \ref{L:Sum} with $n_{1} = 2$, $n_{2} = k + 1$ and $r = 2$. This implies that (recall $D = D_{1} + D_{2}$)
	   	\begin{equation} \label{L:OuterR1:2}
    			2\left(\min_{1 \leq j \leq k} \omega\left(\by^{j}\right) - \omega\left(\bx'\right)\right)\sum_{j = 1}^{k} \eta_{j} \leq 2\sum_{j = 1}^{k} \eta_{j}\left(\omega\left(\by^{j}\right) - \omega\left(\bx'\right)\right) \leq D_{1}^{2} + D_{2}^{2} \leq 2^{2}.
    		\end{equation}
		Now, we use classical properties of the harmonic series, to obtain that
		\begin{equation*}
			\sum_{j = 1}^{k} \eta_{j} = c\sum_{j = 1}^{k} \left(j + 1\right)^{-1} = c\left(\sum_{j = 0}^{k} \left(j + 1\right)^{-1} - 1\right) > c\left(\sum_{j = 1}^{k} j^{-1} - 1\right) \geq c\left(\ln\left(k + 1\right) - 1\right),
		\end{equation*}		
		and by combining it with \eqref{L:OuterR1:2}, we obtain the desired result.
	\end{proof}		
	We conclude this section with the following convergence result.
	\begin{proposition}[Convergence of Bi-SG]
		Let $\Seq{\bx}{k}$ and $\Seq{\by}{k}$ be the two sequences generated by the Bi-SG Algorithm. Then
		\begin{itemize}
    			\item[$\rm{(i)}$] Let $1/2 < \alpha \leq 1$. Any limit point of $\Seq{\by}{k}$ belongs to $X^{\ast}$, \ie an optimal solution of the inner problem \eqref{Prob:IP}. In particular, we have
    				\begin{equation*}
    					\lim_{k \rightarrow \infty} \dist\left(\by^{k} , X^{\ast}\right) = 0.
    				\end{equation*}
    			\item[$\rm{(ii)}$] Let $1/2 < \alpha < 1$. Any limit point of $\left\{ \by^{k_{best}} \right\}_{k \in \nn}$ belongs to $X'$, \ie an optimal solution of the outer problem \eqref{Prob:OP}. In particular, we have
    				\begin{equation*}
    					\lim_{k \rightarrow \infty} \dist\left(\by^{k_{best}} , X'\right) = 0.
    				\end{equation*}
		\end{itemize}
	\end{proposition}
	\begin{proof}
    		First, from Theorem \ref{T:Bounded} the sequence $\Seq{\by}{k}$ is bounded (see also Lemma \ref{L:Basic}(ii)). We need to show that any convergent sub-sequence $\left\{ \by^{i_{k}} \right\}_{k \in \nn}$ converges to some point in $X^{\ast}$. Thus, we take a certain converging sub-sequence $\left\{ \by^{i_{k}} \right\}_{k \in \nn}$. We denote its limit point with $\hat{\by}$. From Theorem \ref{T:InnerRate}, we have that
		\begin{equation*}
        		\lim\limits_{k \rightarrow \infty} \varphi\left(\by^{k}\right) = \varphi^{\ast},
		\end{equation*}
		where $\varphi^{\ast}$ is the optimal value of the inner problem \eqref{Prob:IP}. Thus, by using the fact that $\varphi$ is lower semi-continuous, we have that
		\begin{equation*}
	        \varphi\left(\hat{\by}\right) \leq \liminf\limits_{k \rightarrow \infty} \varphi\left(\by^{i_{k}}\right) = \lim\limits_{k \rightarrow \infty} \varphi\left(\by^{k}\right) = \varphi^{\ast},
		\end{equation*}
		which proves that $\hat{\by} \in X^{\ast}$, as desired.
\medskip

		The second item follows similar arguments on the bounded sequence $\left\{ \by^{k_{best}} \right\}_{k \in \nn}$, where Theorem \ref{T:InnerRate} is replaced with Theorem \ref{T:SimuRate}.
	\end{proof}

\subsection{Rate of Convergence of Bi-SG - Strongly Convex Case}
	In this part, we show that Bi-SG with choice (V2) of the mapping $\Omega_{k} \equiv G_{\eta_{k}}^{\sigma , \psi}$, $k \in \nn$, guarantees a better convergence rate in terms of the outer objective function $\omega$. In order to derive the improved convergence rate, we will make the following assumption (in addition to the previous Assumptions \ref{A:Composite}, \ref{A:Outer} and \ref{A:Outer2} (C2)).
	\begin{assumption} \label{A:Strong}
    			\item[$\rm{(i)}$] $\sigma$ is additionally $\beta$-strongly convex ($\beta > 0$).
    			\item[$\rm{(ii)}$] $g$ has in addition bounded sub-gradients on bounded subsets of $\real^{n}$.
	\end{assumption}
	We will prove the following rate of convergence result.
	\begin{theorem}[An outer linear rate of convergence] 
		Let $\Seq{\bx}{k}$ and $\Seq{\by}{k}$ be the two sequences generated by the Bi-SG Algorithm. For any $\bx' \in X'$ and $k \geq 1$, we have
		\begin{equation*}
    			\varphi\left(\bx^{\tilde{k}}\right) - \varphi\left(\bx'\right) \leq \frac{2H + D^{2}}{k^{\alpha}} \qquad \text{and} \qquad \omega\left(\bx^{\tilde{k}}\right) - \omega\left(\bx'\right) \leq \left(\frac{1}{e^{c\beta/4}}\right)^{k^{1 - \alpha}}D^{2},
		\end{equation*}
		where $\tilde{k} = \argmin \left\{ \omega\left(\bx^{j}\right) : \, k + 1 \leq j \leq 2k \right\}$. 
	\end{theorem}
	\begin{proof}
		Let $k \geq 1$. We first prove the inner rate. To this end, we take $\xi \in \partial \varphi\left(\bx^{\tilde{k}}\right)$ and from Assumption \ref{A:Strong}(ii) we have that $\norm{\xi} \leq D_{2}$ (we use here $D_{2}$ just for the simplicity). Hence, by applying the sub-gradient inequality on the convex function $\varphi$, we obtain (recall \eqref{U:Outer})
		\begin{equation*}
			\varphi\left(\bx^{\tilde{k}}\right) - \varphi\left(\by^{\tilde{k} - 1}\right) \leq \act{\xi , \bx^{\tilde{k}} - \by^{\tilde{k} - 1}} = -\eta_{\tilde{k} - 1}\act{\xi , \Omega_{\tilde{k} - 1}\left(\by^{\tilde{k} - 1}\right)} \leq \eta_{\tilde{k} - 1}D^{2} \leq \frac{D^{2}}{k^{\alpha}},
		\end{equation*}		 
		where the second inequality follows from Theorem \ref{T:Bounded} and the fact that $\norm{\xi} \leq D_{2} < D$, while the last inequality follows from the fact that $\eta_{\tilde{k} - 1} \leq k^{-\alpha}$ (since $k \leq \tilde{k} - 1$ and $c \leq 1$). Now, from Theorem \ref{T:InnerRate} we obtain that 
	    \begin{equation*}
        		\varphi\left(\bx^{\tilde{k}}\right) - \varphi\left(\bx'\right) = \varphi\left(\by^{\tilde{k} - 1}\right) - \varphi\left(\bx'\right) + \varphi\left(\bx^{\tilde{k}}\right) - \varphi\left(\by^{\tilde{k} - 1}\right) \leq \frac{2H}{\left(\tilde{k} - 1\right)^{\alpha}} + \frac{D^{2}}{k^{\alpha}} \leq \frac{2H + D^{2}}{k^{\alpha}},
    		\end{equation*}   
    		where the second inequality follows from the fact that $k \leq \tilde{k} - 1$.
\medskip
		
		Now, we prove the outer rate. We split the proof into two cases. First, we assume that there exists $j \in \left\{ k + 1 , k + 2 , \ldots , 2k \right\}$ such that $\omega\left(\bx^{j}\right) \leq \omega\left(\bx'\right)$. In this case, the result directly follows from the definition of $\tilde{k}$. Second, if otherwise, we take an integer $j \in \left\{ k , k + 1 , \ldots , 2k - 1 \right\}$ and, by using Proposition \ref{P:ProximalInequality} with $s = \sigma$ $h = \psi$, $t = \eta_{j}$, $\bx = \by^{j}$ and $\by = \bx'$ (see \ref{P:ProximalInequality:1}), we obtain
    		\begin{equation*}
         	\omega\left(\bx'\right) - \omega\left(\bx^{j + 1}\right) \geq \frac{1}{2\eta_{j}}\left(\norm{\bx^{j + 1} - \bx'}^{2} - \norm{\by^{j} - \bx'}^{2}\right) + \ell_{\sigma}\left(\bx' , \by^{j}\right).
	    \end{equation*}
    		In addition, from the strong convexity of $\sigma$ (see Assumption \ref{A:Strong}(i)), we have that 
    		\begin{equation*}
        		\ell_{\sigma}\left(\bx' , \by^{j}\right) = \sigma\left(\bx'\right) - \sigma\left(\by^{j}\right) - \act{\nabla \sigma\left(\by^{j}\right) , \bx' - \by^{j}} \geq \frac{\beta}{2}\norm{\by^{j} - \bx'}^{2}.
	    \end{equation*}
    		Thus, by combining these two facts, we obtain that    
    		\begin{equation*}
        		\norm{\bx^{j + 1} - \bx'}^{2} \leq \left(1 - \eta_{j}\beta\right)\norm{\by^{j} - \bx'}^{2} + 2\eta_{j}\left(\omega\left(\bx'\right) - \omega\left(\bx^{j + 1}\right)\right) \leq \left(1 - \frac{c\beta}{2k^{\alpha}}\right)\norm{\bx^{j} - \bx'}^{2},
    		\end{equation*}
    		where the second inequality follows the facts that $\omega\left(\bx^{j + 1}\right) \leq \omega\left(\bx'\right)$ (which follows from the case assumption and the fact that $j + 1 \in \left\{ k + 1 , k + 2 , \ldots , 2k \right\}$) and the definition of $\eta_{j}$ (recall that $\alpha < 1$). Since this is true for any  $j \in \left\{ k , k + 1 , \ldots , 2k - 1 \right\}$, we obtain that
	    \begin{equation*}
        		\norm{\bx^{2k} - \bx'} \leq \left(1 - \frac{c\beta}{2k^{\alpha}}\right)^{k / 2}\norm{\bx^{k} - \bx'} \leq \left(\left(1 - \frac{c\beta}{2k^{\alpha}}\right)^{\frac{2k^{\alpha}}{c\beta}}\right)^{\frac{k^{1 - \alpha}c\beta}{4}}D_{1} \leq \left(\frac{1}{e^{c\beta/4}}\right)^{k^{1 - \alpha}}D_{1},
	    \end{equation*}
    		where the second inequality follows from Theorem \ref{T:Bounded} and the last inequality follows from the fact that $\left(1 - 1/r\right)^{r} \leq e^{-1}$ for any $r > 0$.
\medskip

    		Now, applying the sub-gradient inequality on the convex function $\omega\left(\cdot\right)$ with $\mu \in \partial\omega\left(\bx^{2k}\right)$, we have that   
    		\begin{equation*}
	        \omega\left(\bx^{2k}\right) - \omega\left(\bx'\right) \leq \act{\mu , \bx^{2k} - \bx'} \leq\norm{\mu}\cdot\norm{\bx^{2k} - \bx'} \leq \left(\frac{1}{e^{c\beta/4}}\right)^{k^{1 - \alpha}}D^{2}
    		\end{equation*}
    		where the last inequality follows from Theorem \ref{T:Bounded} and Remark \ref{R:BoundedSub} since  $D_{1}D_{2} < D^{2}$. The result follows now, directly, from the definition of $\tilde{k}$.
	\end{proof}

\section{Numerical Experiments} \label{Sec:Numerical Experiments}
	We have conduct numerical experiments in order to compare several existing algorithms for solving bi-level optimization problems. We have focused on two challenging tasks from the domain of machine learning:
	\begin{itemize}
		\item[$\rm{(i)}$] Fake News Classification.
		\item[$\rm{(ii)}$] Song Release Year Prediction.
	\end{itemize}
	In each case, we will apply the following five methods:
	\begin{itemize}
    	\item[$\rm{(i)}$] Bi-SG (Version II, with $\psi \equiv \omega$ and $\sigma \equiv \bo$), where we considered two values of $\alpha \in \left\{ 0.85 , 0.95\right\}$.
    	\item[$\rm{(ii)}$] BiG-SAM of \cite{SS2017}, where we applied the algorithm on the Moreau envelope of $\omega$ (instead of $\omega$ itself), since it is a non-smooth function. Therefore, it should be noted that BiG-SAM in this case tackles the smoothed bi-level problem and not the original non-smooth bi-level problem. We use two values of the parameter $\delta \in \left\{ 0.01 , 1\right\}$.
    	\item[$\rm{(iii)}$] IR-IG and deterministic aRB-IRG of \cite{AY2019}.
    	\item[$\rm{(iv)}$] FIBA (Fast Iterative Bilevel Algorithm) of \cite{HS2017}, which is a modification of the Fast Iterative Thresholding Algorithm (FISTA) to the bi-level optimization setting. The FIBA algorithm consists of three main sequences of hyper-parameters.
	\end{itemize}	
	 All experiments were run on 11th Gen Intel(R) Core(TM) i7-1195G7 @ 2.90GHz 2.92 GHz with a total RAM memory of 16GB and 4 physical cores (using a free account of Google cola notebook). We run each algorithm for $300$ seconds and provide details about both the inner and outer optimization problems. In terms of the inner optimization problem, for each method $1 \leq l \leq 7$, we record the following measure over time and
	\begin{equation*}
		\Delta^{l} \varphi\left(t\right) = \varphi\left(\bx^{(l)}_{t}\right) - \varphi^{\ast},
	\end{equation*}
	where $\bx_{t}^{(l)}$ is the point computed by the method $l$ after $t$ seconds, and $\varphi^{\ast}$ is the optimal value of $\varphi$, which we compute using an off-the-shelf solver\footnote{$https://scikit-learn.org/stable/modules/generated/sklearn.linear\_model.LogisticRegression.html$}. As for, the outer optimization problem, we present plot that record the progress of the values of $\omega$ versus the $-\log(\Delta\varphi)$.  

\subsection{Fake News Classification}	
	In this experiment, we wish to distinguish between fake and true news articles based on a dataset used in a broad competition held on kaggle website\footnote{see $https://www.kaggle.com/competitions/fake-news/overview$}. This dataset contains 10,413 true  and 10,387 fake labeled news articles, respectively. For each article, the data-set contains a label (that indicates whether it is true or fake) and three fields: `title', `author' and `text'. We use only the fields `title' and `author' (by concatenating them together) and randomly select 4,000 true and false articles. We pre-process the data, which means we stem every word\footnote{Stemming is the process of transforming a word to its root form. For example, $\text{bought} \Rightarrow \text{buy}$} and remove stop-words and all of the words with non-alphabetic characters\footnote{For example, the sentence `I was going to buy a tasty apple' is transformed to `I go buy taste apple'}. Afterwards, we represent each article as a `bag of bi-grams'. Formally, we denote by $A \in \real^{N \times m}$ the features matrix such that $\ba_{i}$, which denotes the $i$-th row of $A$ for any $1 \leq i \leq N$, contains the number of times a bi-gram appears in the $i$-th article, for all bi-grams appearing in the data-set. To demonstrate this idea, suppose that the list of the bi-grams appearing in at least one of the articles (in the concatenation of the field `title' with `author' after prepossessing) is [`he buy', `we cry', `beautiful dress', `beautiful house', `buy beautiful'] and the first article (after prepossessing) is ``he buy beautiful house". In this case, $a_{1} = \left(1 , 0 , 0 , 1 , 1\right)$.
\medskip

	A basic approach in machine learning for this task, is by minimizing the logistic regression function of the given data-set. More precisely, we are given a matrix $A \in \real^{N \times m}$ and a vector $\bz \in \real^{N}$, which contains the labels of the articles, the corresponding logistic loss function is defined by				
	\begin{equation*}
		\varphi\left(\bx\right) \equiv \frac{1}{N}\sum_{i = 1}^{N}\left(\log\left(w\left(\bx^{T}\ba_{i}\right)\right)\bz_{i} + \log\left(1 - w\left(\bx^{T}\ba_{i}\right)\right)\left(1 - \bz_{i}\right) \right),
	\end{equation*}
	where $w : \real \rightarrow \real_{++}$ is the sigmoid function defined by $w\left(t\right) = \exp^{t}/\left(1 + \exp^{t}\right)$ and $\bz = \left(\bz_{1} , \bz_{2} , \ldots , \bz_{N}\right)^{T}$ for which $\bz_{i} \in \left\{0 , 1 \right\}$. It is well-known that since the number of features $m$ is not negligible relative to the number of articles $N$, over-fitting may occur. To avoid it, one popular approach is to regularize the logistic objective function with another function $\omega\left(\cdot\right)$, that is, solving the following optimization problem
	\begin{equation*}
		\min_{\bx \in \real^{m}} \left\{ \varphi\left(\bx\right) + \lambda\omega\left(\bx\right) \right\},
	\end{equation*}
	where $\lambda > 0$ is a regularization parameter. As we discussed in the Introduction this raises the major question of how to determine the ``right" regularization parameter? Here bi-level optimization enters into an action, since it allows to overcome this difficulty by focusing on solving the following minimization
	\begin{equation*}
		\min_{\bx \in \real^{m}} \left\{ \omega\left(\bx\right) : \, \bx \in X^{\ast}\right\},
	\end{equation*}
	where $X^{\ast}$ is the set all minimizers of the logistic regression function $\varphi\left(\cdot\right)$. The bi-level approach can be used here to reduce the number of features in the vector $\bx$ obtained from minimizing the logistic loss $\varphi\left(\cdot\right)$. In other words, we wish to find the sparsest vector in between all the minimizers of the logistic regression function $\varphi\left(\cdot\right)$. Therefore, any outer function that promote sparsity could be relevant to achieve this goal. Here, we take the following function
	\begin{equation*}	
    		\omega\left(\bx\right) = \norm{\bx}_{1} + 0.05\norm{\bx}_{2}^{2}.
	\end{equation*}	
	We perform three different experiments, where in the $i$-th experiment, $1 \leq i \leq 3$, we take only the $\bb_{i}$ most frequent bi-grams, where $\bb_{i}$ denotes the $i$-th element of the vector $\bb = \left(1000 , 2500 , 5000\right)^{T}$. In each experiment we applied one of the five methods (Bi-SG, BiG-SAM, IR-IG, aRB-IRG and FIBA) that were mentioned above.
\newpage

 	\begin{figure}[h]
        \centering
        \includegraphics[width = 16cm]{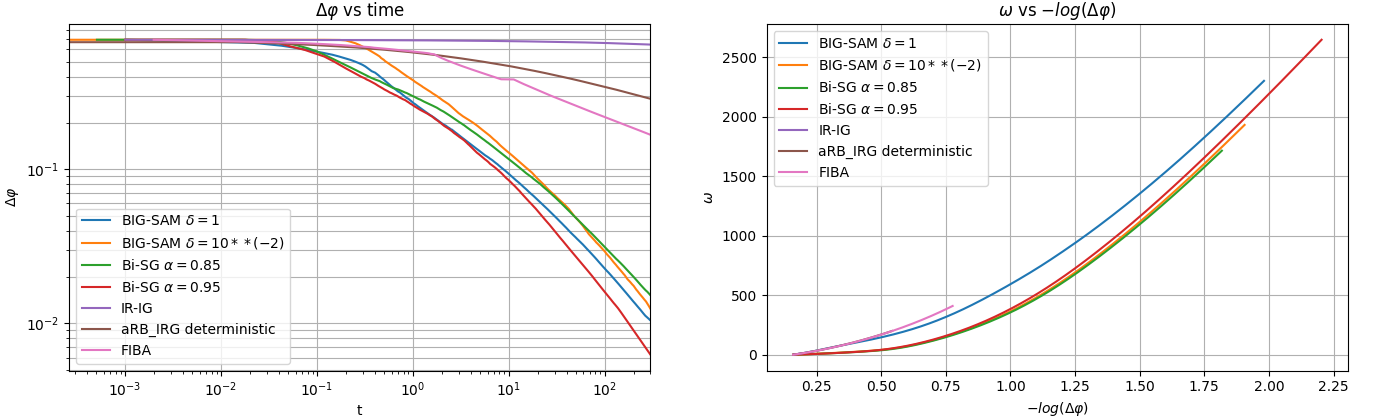}
		\caption{1000 parameters}
    \end{figure}

 	\begin{figure}[h]
        \centering
        \includegraphics[width = 16cm]{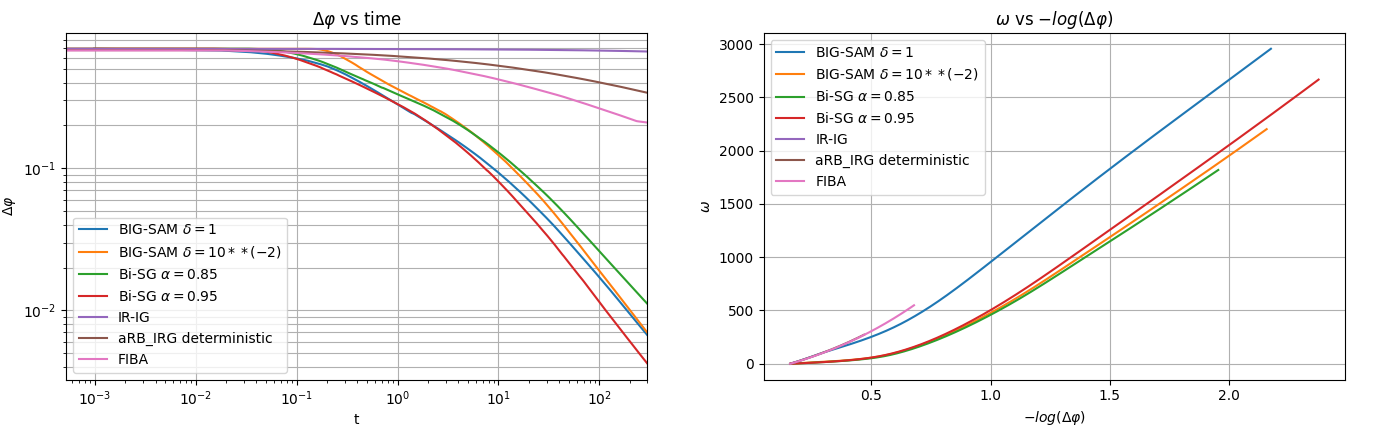}
		\caption{2500 parameters}
    \end{figure}

	Note, in the last experiment, there is an infinite number of optimal solutions to the inner optimization problem, because there are more features than articles. 
\medskip

 	\begin{figure}[h]
        \centering
        \includegraphics[width = 16cm]{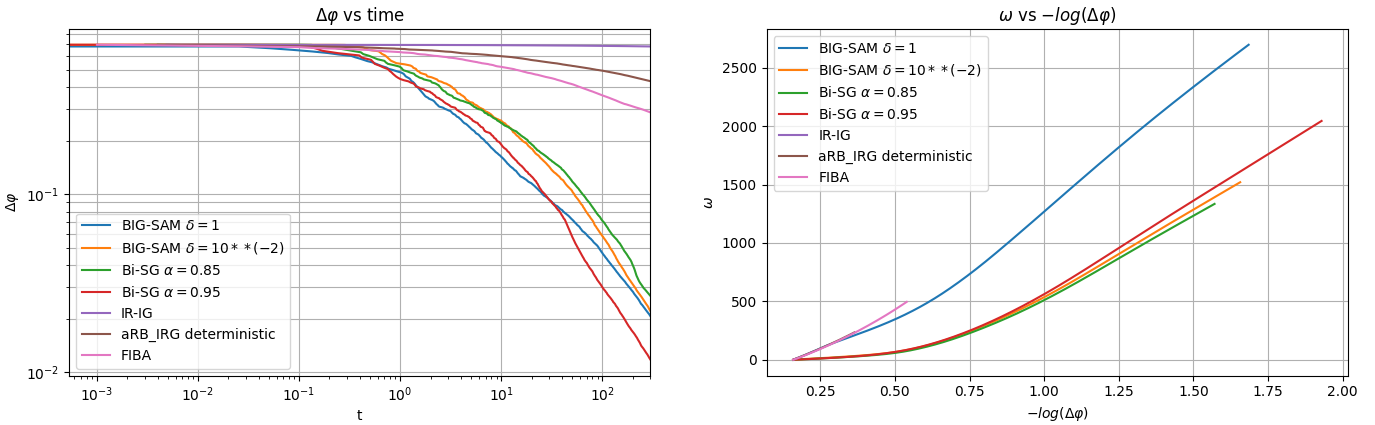}
		\caption{5000 parameters}
    \end{figure}

	A few comments on the numerical results are now in order. From all the figures on the left-hand side, we see that Bi-SG (with $\alpha = 0.95$) achieves the best results in terms on the inner objective function. More precisely, in all experiments, Bi-SG obtains the best optimality gap in terms of $\varphi$. Regarding FIBA, which requires three main sequences of hyper-parameters, when we set them with the original values that were suggested by the authors in \cite{HS2017}, the algorithm did not converge in all the experiments. Recall that in our case both the inner and outer functions are not Lipschitz continuous, as required to guarantee the convergence of FIBA. As a result, we tried different values for the aforementioned sequences until we have tuned them in order to find values for which the algorithm converges. We have witnessed in all the experiments that FIBA exhibited a very slow convergence rates in terms of both the inner and the outer objective functions (it should be noted that FIBA has no proven convergence rates).
\medskip

	Regarding the plots on the right-hand side we first would like to describe the motivation behind comparing the values of the inner and outer objective functions. In the $x$-axis we record the inner objective function optimality gap (after taking minus log), which means that as the value increases it means that the corresponding algorithm is more accurate in terms of solving the inner problem. Therefore, algorithms that were not able to obtain good accuracy in solving the inner problem appear as a short line in these plots (for example, IR-IG and aRB-IRG can't be seen in this plots since they almost didn't decrease the inner optimality gap as can be seen from the plots of the left-hand side). In between the algorithms that achieved a good accuracy in the inner optimality gap, we seek the algorithm that obtains the minimal value of $\omega$. For instance, in the plot of $2500$ parameters we see that if we compare the algorithms when the $x$-axis equals around $2$ then Bi-SG with $\alpha = 0.85$ is better than both versions of BiG-SAM and Bi-SG with $\alpha = 0.95$. However, since Bi-SG with $\alpha = 0.95$ achieves higher accuracy in the inner problem than all other algorithms its final value of $\omega$ is higher.
 
\subsection{Songs Release Year Prediction}
	In this experiment, we used the YearPredictionMSD dataset\footnote{https://archive.ics.uci.edu/ml/datasets/yearpredictionmsd} to train a linear regression model to predict songs release years. The data-set contains information on $515,345$ songs, with a release year ranging from 1992 to 2011. For each song, the data-set contains its release year and additional 90 attributes (about its sound timbres\footnote{Timbre(Pronounced Tam-ber) is the quality of a musical note. It is what makes a musical note sound different from another one. Words like round, brassy, sharp, or bright can be used to describe the timbre of a sound}).  We considered two samples of the data-set with $1000$ and $10,000$ uniformly i.i.d randomly selected songs (without replacements), respectively. For each sample, we use $A$ to denote the feature matrix induced from it. Therefore., the $(i, j)$ entry of $A$ is the value of the $j$-th attribute of the $i$-th song in the sample. In addition, we use $b$ to denote the release year of the songs in the sample. That is, the $i$-th entry in $b$ is the release year of the $i$-th song in the sample. Our goal is to train a predictor by finding a weights vector $x$ that minimizes the following objective function
\begin{equation} \tag{SI}
    \varphi\left(x\right) = \frac{1}{N}(1/2)\norm{Ax - b}^{2},
\end{equation}
where $N$ is the number of songs in the sample.
\medskip

	We also take several steps to pre-process the data. First, we use a min-max scaling technique\footnote{$https://en.wikipedia.org/wiki/Feature_scaling$}, then we concatenate the $\boldsymbol{1} = (1,1,\ldots, 1)^{T} \in \mathbb{R}^{N}$ vector to $A$. Hence, the model could learn the interceptor. In addition, we found out that $A^{T}A$ is PD (positive definitive) meaning that the Problem (SI) has a single optimal solution. In this case, there is no meaning to the bi-level setting. Hence, we decided to modify $A$ so $A^{T}A$ would be semi-positive definitive (PSD), which means that problem (SI) would have multiple optimal solutions by adding $90$ co-linear attributes to $A$. For each additional co-linear attribute, we uniformly randomly i.i.d sampled $10$ attributes from the original matrix $A$, without replacements and let $\Tilde{A}$ be the matrix that contains the values of these $10$ attributes for each song in $A$. Then, the $i$-th row in $\tilde{A}$ is the vector that contains the values of these $10$ attributes for the $i-$th song in the sample. In addition, we generated a vector of i.i.d uniformly randomly sampled values from the interval $[-1, 1]$. We denote this vector with $v \in \mathbb{R}^{10}$. Using $\Tilde{A}$ and $v$, we compute the additional co-linear attribute as $\Tilde{A}v$. This results in that $A$ having co-linear columns (i.e., it has columns that are a linear combination of other columns). Now, our goal is to solve the following outer-level problem
\begin{equation} \tag{SO}
    \begin{aligned}
    \min_{x \in \mathbb{R}^{n}} \quad & \norm{x}_{1} + 0.05\norm{x}^{2} \\
    \textrm{s.t} \quad & x \in \argmin_{y \in \mathbb{R}^{n}} \varphi(y),
    \end{aligned}
\end{equation}
	which means that in this case $\omega = \norm{x}_{1} + 0.05\norm{x}$. Our goal is to use different algorithms to solve problem (SO) and compare them in terms of both the inner and outer objective functions values. In order to achieve this goal, we computed and compared the values of $\varphi$ and $\omega$ over time for all the algorithms mentioned above. We compute the optimal value of the inner function using an off-the-shelf solver\footnote{Visit $http://scikit-learn.org/stable/modules/generated/sklearn.linear_model.LinearRegression.html$} and denote it with $\varphi^{\ast}$. We used the same experimental configurations as in the previous experiment. The results are recorded next.
	
 	\begin{figure}[h]
        \centering
        \includegraphics[width = 16cm]{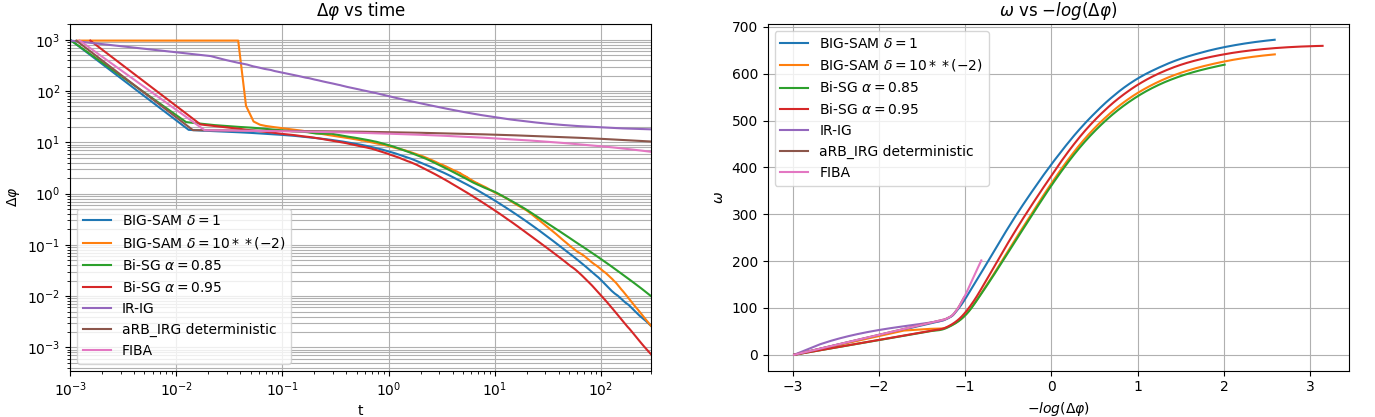}
		\caption{1000 parameters}
    \end{figure}

 	\begin{figure}[h]
        \centering
        \includegraphics[width = 16cm]{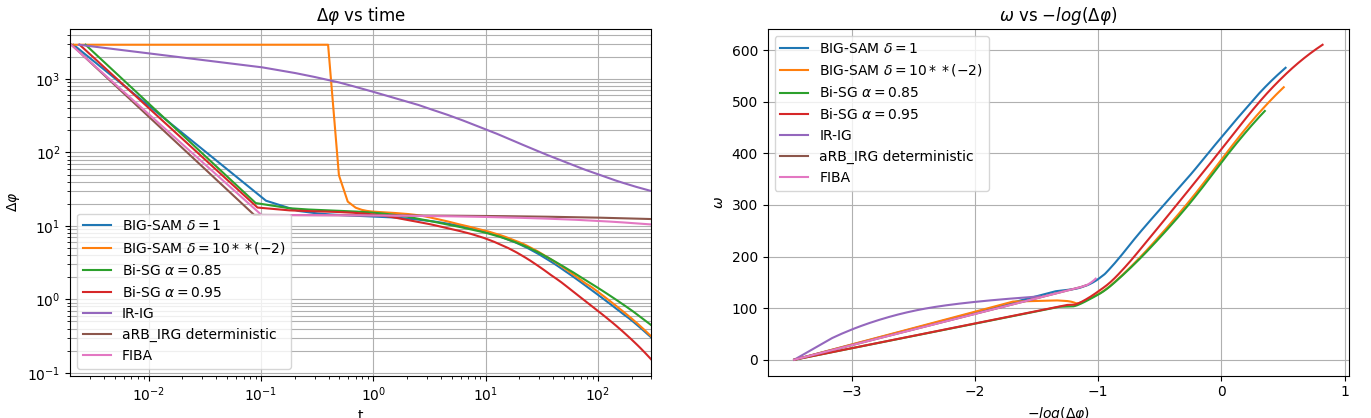}
		\caption{10000 parameters}
    \end{figure}

	In this experiment, we also see that Bi-SG with $\alpha = 0.95$ is the best algorithm in terms of the inner objective function. Regarding the plots on the right-hand side, we see a similar situation as in the previous experiment, since Bi-SG with $\alpha = 0.85$ achieves the lowest value of $\omega$ among the algorithms that obtain a good accuracy in terms of the inner objective function.
	
\section{Conclusion}
	In this paper, we developed the first-order algorithm Bi-SG for solving bi-level optimization problems, where the inner objective function $\varphi$, is of a classical composite structure, and that the outer objective function $\omega$ satisfies several mild assumptions, that do not include differentiabilty nor strong convexity. The Bi-SG algorithm enjoys convergence rates in terms of both the inner and outer objective functions values. In addition, we analyzed Bi-SG in the special setting where the outer objective function $\omega$ is strongly convex. In this case, we proved that Bi-SG enjoys a linear rate of convergence in terms of the outer objective function. Moreover, we compared Bi-SG with a few state-of-the-art algorithms for bi-level optimization problems by applying them on diverse types of machine learning problems.
\clearpage
\bibliographystyle{plain}
\bibliography{notes}
\end{document}